\documentclass[a4paper, 11pt]{article}\usepackage{amsmath,amssymb,exscale,amsthm}\usepackage{a4wide} 

\usepackage[english]{babel}
\usepackage{graphicx,psfrag}
\usepackage[all,cmtip]{xy}

\begin{document}
\selectlanguage{english}

\title{Subvarieties of moduli spaces of sheaves via finite coverings}
\author{Markus Zowislok\footnote{Department of Mathematics, Imperial College London, 180 Queen's Gate, London SW7 2AZ, UK}}\date{}\maketitle\newtheorem{theorem}{Theorem}[section]\newtheorem{definition}[theorem]{Definition}\newtheorem{corollary}[theorem]{Corollary}\newtheorem{lemma}[theorem]{Lemma}\newtheorem{proposition}[theorem]{Proposition}\newtheorem{example}[theorem]{Example} \newcommand{\acks}{\vspace{0.2cm}\noindent\textit{Acknowledgements. }}\newcommand{\titar}{\textit} 


\begin{abstract}
Given a finite unbranched covering of a nonsingular projective scheme we analyse the morphism between moduli spaces of sheaves induced by pullback.
We have a closer look at cyclic coverings and, in particular, at canonical coverings of surfaces.
Our main application is the construction of Lagrangian subvarieties of certain irreducible holomorphically symplectic manifolds that arise from moduli spaces of sheaves on K3 or abelian surfaces.
\end{abstract}

\maketitle

\newcommand{\IQ}{\ensuremath{\mathbb{Q}}}\newcommand{\IR}{\ensuremath{\mathbb{R}}}\newcommand{\IZ}{\ensuremath{\mathbb{Z}}}
\newcommand{\IP}{\ensuremath{\mathbb{P}}}\newcommand{\IC}{\ensuremath{\mathbb{C}}}\newcommand{\IN}{\ensuremath{\mathbb{N}}}
\newcommand{\cO}{\ensuremath{\mathcal{O}}}\newcommand{\cH}{\ensuremath{\mathcal{H}}}
\newcommand{\bc}{\ensuremath{\mathrm{c}}}\newcommand{\bH}{\ensuremath{\mathrm{H}}}
\newcommand{\ext}{\ensuremath{\mathrm{ext}}}\newcommand{\Ext}{\ensuremath{\mathrm{Ext}}}\newcommand{\cExt}{\ensuremath{\mathcal{E}xt}}
\renewcommand{\hom}{\ensuremath{\mathrm{hom}}}\newcommand{\Hom}{\ensuremath{\mathrm{Hom}}}
\newcommand{\rk}{\ensuremath{\mathrm{rk\,}}}\newcommand{\codim}{\ensuremath{\mathrm{codim}}}
\newcommand{\Coh}{\ensuremath{\mathrm{Coh}}}
\newcommand{\Pic}{\ensuremath{\mathrm{Pic}}}\newcommand{\Num}{\ensuremath{\mathrm{Num}}}
\newcommand{\CH}{\ensuremath{\mathrm{CH}}}\newcommand{\NS}{\ensuremath{\mathrm{NS}}}
\newcommand{\Amp}{\ensuremath{\mathrm{Amp}}}
\newcommand{\Hev}{\ensuremath{\Lambda(Y)}}\newcommand{\Hpev}{\ensuremath{\Lambda'(Y)}}
\newcommand{\Hilb}{\ensuremath{\mathrm{Hilb}}}
\newcommand{\lel}{\ensuremath{\;(\le)\;}}\newcommand{\lelo}{\ensuremath{\;(\le_0)\;}}
\newcommand{\lcm}{\ensuremath{\mathrm{lcm}}}\newcommand{\gcf}{\ensuremath{\mathrm{gcd}}}
\newcommand{\ord}{\ensuremath{\mathrm{ord}}}
\newcommand{\qede}{\vspace{-1.3cm}\[\qedhere\]}
\renewcommand{\labelenumi}{(\arabic{enumi})\ }\renewcommand{\labelenumii}{(\alph{enumii})\ }

\newtheorem{question}[theorem]{Question}\newtheorem{conjecture}[theorem]{Conjecture}

\section{Introduction}

Is Gieseker stability preserved under pullback by a finite covering? How are the corresponding moduli spaces related?
In this article, we exhibit the behaviour of stability under pullback by a finite unbranched covering of a nonsingular projective scheme.
The roots of this question go back to Kim's article \cite{Kim98a}, which is on the canonical covering of an Enriques surface and $\mu$-stable sheaves.

Our main results hold not only for the notion of Gieseker stability, but also for twisted stability and for $(H,A)$-stability. Therefore we will always write (semi)stable whenever the statement allows all three notions. These stability notions are recalled in Section \ref{StabNot}.

For the whole article let $X$ be a nonsingular projective irreducible variety over $\IC$, $G$ a finite group acting freely on $X$, $f\colon X\to Y$ the quotient of $X$ by $G$, and $H$ an ample divisor on $Y$.
In Section \ref{arbdim} we show that the pullback $f^*E$ of a semistable sheaf $E$ on $Y$ is again semistable (Proposition \ref{PbStab}).
After some more precise results on the behaviour of the stability property, we apply these results to moduli spaces:\\

\noindent \textbf{Theorem \ref{MHApb}}. \titar{Let $P$ be a polynomial, $M_Y$ a quasiprojective and nonempty subscheme of the moduli space $M_Y(P)$ of semistable sheaves on $Y$ with (twisted) Hilbert polynomial $P$, $M_X=M_X(\deg f\cdot P)$ the moduli space of semistable sheaves on $X$ with (twisted) Hilbert polynomial $\deg f\cdot P$, and $M_Y^s$ and $M_X^s$ the respective loci of stable sheaves.
The pullback by $f$ induces a morphism $f^*\colon M_Y\to M_X$ which maps closed points $[E]$ to $[f^*E]$.
The closed points of its image are represented by polystable $G$-sheaves, and $(f^*)^{-1}( M_X^s )\subseteq M_Y^s\,.$}\\

The restriction to a cyclic covering given by a line bundle $L$ of finite order in Section \ref{arbCC} allows a deeper analysis of the pullback morphism. The main tool is the group action on the moduli space of sheaves on $Y$ induced by tensoring with $L$. This allows us to give a precise description of the locus of stable sheaves becoming strictly semistable in Theorem \ref{CycStab}.\\

Section \ref{Surfaces} contains the application to canonical coverings of surfaces. We investigate the double covering $X\to Y$ of an Enriques surface $Y$ by a K3 surface $X$ as well as the canonical covering $X\to Y$ of a bielliptic surface $Y$ by an abelian surface $X$.
The main results are given in Theorems \ref{CC} and \ref{CCsst}.
An interesting question is in what cases there are Lagrangian subvarieties as images of $f^*$ inside the moduli spaces of semistable sheaves on $X$, and what kind of varieties these subvarieties are. 
If $X$ is a K3 surface and $f^*u$ is primitive then the moduli space $M_X(f^*u)$ in general is an irreducible symplectic manifold, and
whenever there is a suitable stable sheaf on $Y$, one gets a Lagrangian subvariety as described in Proposition \ref{CanDouSt}.
If $X$ is an abelian surface, the situation is different: in order to produce higher dimensional irreducible symplectic manifolds out of moduli spaces of sheaves on $X$ one has to get rid of superfluous factors in the Bogomolov decomposition by taking a fibre of the Albanese map. The last part of Section \ref{LIHS} explains how and why this reduction reduces the Lagrangian subvarieties to (smaller) Lagrangian subvarieties in the case of a double covering, resulting in Proposition \ref{CanDouStK}.

The most prominent Lagrangian subvarieties of irreducible holomorphically symplectic manifolds occur as fibres or sections of Lagrangian fibrations. Recently, smooth Lagrangian tori have attracted a lot of interest, sparked by the following question of Beauville \cite{Beau10}:
If an irreducible holomorphically symplectic manifold $M$ contains a smooth Lagrangian torus, is this a fibre of a Lagrangian fibration on $M$? 
This has been answered positively combining \cite{GLR11}, \cite{HW12}, and \cite{Mat12}.

It would be interesting to understand better, which kind of Lagrangian subvarieties can be obtained by Propositions \ref{CanDouSt} and \ref{CanDouStK}. 
Some concrete results are contained in Section \ref{example}.
Unfortunately, our knowledge on moduli spaces of sheaves on Enriques and bielliptic surfaces is quite limited at the moment.
However, we expect that complex tori do not occur as Lagrangian subvarieties in the case of sheaves of odd rank. In particular, they do not occur as image under pullback of the moduli space of sheaves of rank one on Enriques surfaces, i.e.\ in the Hilbert scheme case, as explained in Section \ref{example}.
The case of even rank seems to be more promising, as an example of Hauzer can be used to produce an elliptic curve as Lagrangian subvariety in a K3 surface, as explained in Section \ref{example} as well. Hence there is hope that higher-dimensional examples of Lagrangian tori inside symplectic varieties may occur.

On the other side, it is interesting in its own to construct different Lagrangian subvarieties and, in particular, to study their intersection. In \cite{BF09} the authors construct a Gerstenhaber algebra structure and a compatible Batalin-Vilkovisky module structure
giving rise to a de Rham type cohomology theory for Lagrangian intersections.

Shortly after the first version of this article appeared on arXiv.org, an independent article by Sacc\`a \cite{Sac12} on the case of one-dimensional sheaves on Enriques surfaces and the induced pullback by the canonical double covering appeared there as well. Her results are related to my Question \ref{OSgen}.

Finally we give an outlook in Section \ref{Outlook} considering the case of surfaces of general type,
and an appendix briefly recalls the notion of a general ample divisor which occurs in Section \ref{Surfaces}.\\

\acks
{\small The author would like to express his gratitude to S\"onke Rollenske for fruitful discussions and helpful suggestions.
He thanks Arvid Perego for directing his attention to Kim's articles and Hauzer's article and Manfred Lehn for useful comments.
Moreover, he thanks the DFG for the SFB 701, which provides a vibrant research atmosphere at Bielefeld.
This article was completed during a stay at Imperial College London supported by a DFG research fellowship (Az.: ZO 324/1-1).}

\section{Pullback by finite \'etale coverings and stability}\label{arbdim}

In this section we recall the considered notions of stability of sheaves, analyse their behaviour under pullback by $f$ and apply the results to get our general Theorem \ref{MHApb} on the moduli spaces of sheaves.
We assume familiarity with the material presented in \cite{HL10} and use the notation therein.

\subsection{Three stability notions}\label{StabNot}

Our main results hold for the notions of Gieseker stability, twisted stability and $(H,A)$-stability.
Twisted stability and $(H,A)$-stability are two generalisations of Gieseker stability, with an overlap in the case of a surface, see \cite[Corollary 6.2.6]{Zow10}.
We briefly recall the definitions.

\begin{enumerate}
\item \textit{Gieseker stability.} 
A detailed treatment of this notion can be found e.g.\ in \cite[Section 1.2]{HL10}.
Let $H$ be an ample divisor on $Y$ and $E$ a nontrivial coherent sheaf on $Y$.
The Hilbert polynomial of $E$ is $P_H(E)(n):=\chi(E\otimes H^{\otimes n})$. Its leading coefficient multiplied by $(\dim E)!$ is called multiplicity of $E$ and denoted here by $\alpha^H(E)$. It is always positive, and $p_H(E)(n):=\frac{\chi(E\otimes H^{\otimes n})}{\alpha^H(E)}$ is called reduced Hilbert polynomial of $E$.
With these at hand, one says that $E$ is $H$-(semi)stable if $E$ is pure and for all nontrivial proper subsheaves $F\subseteq E$ one has that $p_H(F) \lel p_H(E)$, i.e.\ $p_H(F)(n) \lel p_H(E)(n)$ for $n\gg 0$. 

In order to avoid case differentiation for stable and semistable sheaves we here follow the Notation 1.2.5 in \cite{HL10} using bracketed inequality signs, e.g.\ an inequality with $\lel$ for (semi)stable sheaves means that one has $\le$ for semistable sheaves and $<$ for stable sheaves.
\item \textit{Twisted stability as defined in \cite[Definition 4.1]{Yos03b}.}
Let $H$ be an ample divisor on $Y$, $V$ a locally free sheaf on $Y$ and $E$ a nontrivial coherent sheaf on $Y$.
The $V$-twisted Hilbert polynomial is $P^V_H(E)(n):=\chi(E\otimes V\otimes H^{\otimes n})$, and
the reduced $V$-twisted Hilbert polynomial is $$p^V_H(E)(n):=\frac{\chi(E\otimes V\otimes H^{\otimes n})}{\alpha^H(E\otimes V)}\,.$$
One says that $E$ is $V$-twisted $H$-(semi)stable if $E$ is pure and for all nontrivial proper subsheaves $F\subseteq E$ one has that $p^V_H(F) \lel p^V_H(E)$.

\item \textit{$(H,A)$-stability as defined in \cite[Definition 7.1]{Zow12}.}
Let $H$ and $A$ be two ample divisors on $Y$ and $E$ a nontrivial coherent sheaf on $Y$.
We defined
$$
P_{H,A}(E)(m,n):=\chi(E\otimes H^{\otimes m} \otimes A^{\otimes n}) \quad\mathrm{and}\quad
p_{H,A}(E):=\frac{P_{H,A}(E)}{\alpha^H(E)} \,.
$$
These are polynomials in $m$ and $n$ with degree $d:=\dim E$ in $n$ and $m$ and total degree $d$, and one has $P_{H,A}(E)(\bullet,0)=P_{H}(E)$ and $p_{H,A}(E)(\bullet,0)=p_{H}(E)$.

There is a natural ordering of polynomials in one variable given by the lexicographic ordering of their coefficients.
This generalises to polynomials of two variables by the identification $\IQ[m,n]=(\IQ[m])[n]$, i.e.\ we consider the elements as polynomials in $n$ and use the ordering of $\IQ[m]$ for comparing coefficients.

We introduce another ordering on $\IQ[m,n]$ by defining
$$f\le_0 g\quad:\Leftrightarrow\quad (f(\bullet,0),-f)\le(g(\bullet,0),-g)$$ for $f,g\in\IQ[m,n]$, where on the right hand side we use lexicographic ordering on the product $\IQ[m]\times\IQ[m,n]$, i.e.\ $f\le_0 g$ if and only if $f(\bullet,0)<g(\bullet,0)$ or $f(\bullet,0)=g(\bullet,0)$ and $f\ge g$.
Clearly one has $f=_0g$ if and only if $f=g$.

We say that $E$ is $(H,A)$-(semi)stable if it is pure and if for any proper nontrivial subsheaf $F\subset E$ one has that $p_{H,A}(F)\lelo p_{H,A}(E)\,.$
\end{enumerate}
The case of Gieseker stability can be regained by $V=\cO_Y$ from twisted stability or by $H=A$ from $(H,A)$-stability.
We will always write (semi)stable whenever the statement allows all three notions.
For twisted stability, we will always tacitly assume to have chosen a locally free sheaf $V$ on $Y$, and for $(H,A)$-stability to have chosen an additional ample divisor $A$ on $Y$.

If $H$ and $A$ are two ample divisors on $Y$ and $V$ is a locally free sheaf on $Y$, then the divisors $f^*H$ and $f^*A$ are ample divisors on $X$, and $f^*V$ is a locally free sheaf on $X$.
Also for a sheaf on $X$, we will write (semi)stable instead of Gieseker $f^*H$-(semi)stable, $f^*V$-twisted $f^*H$-(semi)stable or $(f^*H,f^*A)$-(semi)stable.
Moreover, we will denote the usual reduced Hilbert polynomial $p_H(E)$, the reduced $V$-twisted Hilbert polynomial $p_H^V(E)$ and the polynomial $p_{H,A}(E)$ by $p(E)$, and analogously for a sheaf $E'\in\Coh(X)$ one might insert $p_{f^*H}(E')$, $p_{f^*H}^{f^*V}(E')$ or $p_{f^*H,f^*A}(E')$ according to the notion one is interested in.
If we compare the polynomials of $E$ and of a nontrivial subsheaf $F\subseteq E$, then e.g.\ $p(F) \le p(E)$ has to be understood as $p_H(F) \le p_H(E)$, $p_H^V(F) \le p_H^V(E)$, and $p_{H,A}(F) \le_0 p_{H,A}(E)$ (!), respectively. Whenever we need to be more precise, we use the more explicit notation.

\subsection{$f^*$ preserves pureness}

As pureness is part of the definition of stability, we need to check its preservation under pullback.

\begin{lemma}\label{PurePbInv}
Let $E$ be a coherent sheaf on $Y$. $E$ is pure if and only if $f^*E$ is pure.
\end{lemma}
\begin{proof}
For a coherent sheaf $E$ on $Y$ the authors of \cite{HL10} define in Definition 1.1.7 the dual sheaf of $E$ to be $E^D:=\cExt^c(E,K_Y)$, where $K_Y$ is the canonical bundle and $c:=\dim Y-\dim E$ is the codimension of the sheaf $E$. This dualisation commutes with pullback:
By \cite{EGA31} (13.3.5) there is a canonical isomorphism
$$f^*\cExt^c_Y(E,K_Y) \cong \cExt^c_X(f^*E,f^*K_Y)\;.$$
By \cite[Section I.16]{BHPV} one has $f^*K_Y=K_X$, thus $f^*(E^D) \cong (f^*E)^D$.

Now let $E$ be pure. Then by \cite[Lemma 1.1.10]{HL10} the natural homomorphism $\theta_E\colon E\to E^{DD}$ is injective.
As $f^*$ is exact, $f^*(\theta_E)\colon f^*E\to f^*(E^{DD})$ is injective as well.
Using the isomorphism $f^*(E^{DD}) \cong (f^*E)^{DD}$ from above one gets the injectivity of the natural morphism $\theta_{f^*E}\colon f^*E\to (f^*E)^{DD}$, hence $f^*E$ is pure.

Conversely, let $E$ be not pure. Then the natural homomorphism $\theta_E\colon E\to E^{DD}$ has a nontrivial kernel. The exactness of $f^*$ gives a nontrivial kernel of $\theta_{f^*E}\colon f^*E\to (f^*E)^{DD}$, hence $f^*E$ is not pure.
\end{proof}

\subsection{$f^*$ and (semi)stability}

Let $H$ be an ample divisor on $Y$ and $E$ a coherent sheaf on $Y$.
In order to exhibit stability under pullback, we need the behaviour of the reduced Hilbert polynomial under pullback.
\begin{lemma}\label{rHpPb}
$p(f^*E) = p(E)\,.$
\end{lemma}
\begin{proof}
By \cite[\S 12 Theorem 2]{Mum70} one has $\chi(f^*E)=\deg f\; \chi(E)$. Thus on the one hand one has $p_{f^*H}^{f^*V}(f^*E) = p_H^V(E)$,
and on the other hand, one has $p_{f^*H,f^*A}(f^*E) = p_{H,A}(E)$.
\end{proof}

\begin{proposition}\label{PbStab}
$f^*E$ is semistable if and only if $E$ is semistable. If $f^*E$ is stable, then $E$ is stable.
More precisely,
\begin{enumerate}
\item 
$f^*E$ is $f^*H$-semistable if and only if $E$ is $H$-semistable.
If $f^*E$ is $f^*H$-stable, then $E$ is $H$-stable.
\item 
$f^*E$ is $f^*V$-twisted $f^*H$-semistable if and only if $E$ is $V$-twisted $H$-semistable.
If $f^*E$ is $f^*V$-twisted $f^*H$-stable, then $E$ is $V$-twisted $H$-stable.
\item 
$f^*E$ is $(f^*H,f^*A)$-semistable if and only if $E$ is $(H,A)$-semistable.
If $f^*E$ is $(f^*H,f^*A)$-stable, then $E$ is $(H,A)$-stable.
\end{enumerate}
\end{proposition}
\begin{proof}
The pureness condition is given by Lemma \ref{PurePbInv}.
Let $f^*E$ be (semi)stable, and let $F\subseteq E$ be a nontrivial proper subsheaf.
Then $f^*F\subseteq f^*E$ is a nontrivial proper subsheaf, hence by (semi)stability one has 
$p(f^*F)\lel p(f^*E)\,.$
Lemma \ref{rHpPb} yields $p(F)\lel p(E)\,,$ thus $E$ is (semi)stable.

In order to prove the other direction, let $E$ be semistable and assume that $f^*E$ is not semistable.
Let $F\subseteq f^*E$ be a maximal destabilising subsheaf, i.e.\ $F\subseteq f^*E$ is the first part of the Harder-Narasimhan filtration with respect to the considered semistability notion. 
$f^*E$ is a $G$-sheaf in the sense of \cite[\S 7]{Mum70}. As $F$ is maximal, it is also $G$-invariant and thus a $G$-subsheaf of $f^*E$.
By \cite[\S 7 Proposition 2]{Mum70} one has $F\cong f^*\left((f_*F)^G\right)$ and $\left(f_*f^*E\right)^G \cong E$. In particular, $(f_*F)^G$ is isomorphic to a subsheaf of $E$.
The semistability of $E$ yields
$p((f_*F)^G) \le p(E)\,,$ and Lemma \ref{rHpPb} then
$p(F) \le  p(f^*E)\,,$
which contradicts the assumption. Thus $f^*E$ is semistable.
\end{proof}

\noindent Clearly, if $E$ is strictly semistable, i.e.\ semistable but not stable, then $f^*E$ is strictly semistable. 
The converse does not hold for stability: If $E$ is stable then $f^*E$ need not be stable. The following lemma and proposition should be well-known at least for Gieseker stability.

\begin{lemma}\label{StabDestab}
Let $E$ be semistable, and $F_1$ and $F_2$ two destabilising subsheaves with $F_1\cap F_2\neq 0$. Then $p(F_1+F_2)=p(E)$.

Moreover, if $F_1$ is stable, then $F_1\cap F_2 = F_1$ and $F_1 + F_2= F_2$. If $F_2$ is stable as well, then $F_1=F_2$.
\end{lemma}
\begin{proof}
As $F_1$ and $F_2$ are destabilising subsheaves, one has $p(F_1)=p(E)=p(F_2)$.
In particular, $F_1$ and $F_2$ are both semistable, as well as $F_1\oplus F_2$.
Hence one gets the inequality chain $p(F_1\oplus F_2)\le p(F_1+F_2) \le p(E)$,
and together with $p(E)=p(F_1\oplus F_2)$ one has equality everywhere.
The exact sequence
$$0\longrightarrow F_1\cap F_2 \longrightarrow F_1\oplus F_2 \longrightarrow F_1 + F_2 \longrightarrow 0$$
yields
$p(F_1\cap F_2)=p(F_1\oplus F_2)=p(F_1)\,.$
If $F_1$ is stable, then this implies $F_1\cap F_2 = F_1$ and $F_1 + F_2= F_2$.
If $F_2$ is stable as well, then $F_2\cap F_1 = F_2$ as we just proved.
\end{proof}

\begin{proposition}\label{PStGshDe}
Let $E'$ be a semistable sheaf on $X$ and $(F_i)_{i\in I}$ a family of destabilising stable subsheaves of $E'$.
Then there is a subset $J\subseteq I$ such that $E'\supseteq\sum_{i\in I} F_i=\oplus_{i\in J} F_i$.

In particular, if $E'$ is a semistable $G$-sheaf and $F\subseteq E'$ is a destabilising stable subsheaf, then one has $E'\supseteq\sum_{g\in G} g^*F=\oplus_{g\in G'} g^*F$ for a suitable subset $G'\subseteq G$.
\end{proposition}
\begin{proof}
This follows from Lemma \ref{StabDestab} by induction. More precisely, let $F_{1,2}$ be two destabilising stable subsheaves of $E'$. Then either $F_1\cap F_2=0$, i.e.\ $F_1+F_2=F_1\oplus F_2$, or $F_1\cap F_2\neq 0$ and thus $F_1+F_2=F_2$ by Lemma \ref{StabDestab}. This also holds if $F_2$ is not stable, which proves the inductive step.
\end{proof}

\noindent In the preceding Proposition \ref{PStGshDe} the pullbacks $f^*H$, $f^*A$ and $f^*V$, which are hidden in the notation, can be replaced by ample divisors $H'$ and $A'$ on $X$ and a locally free sheaf $V'$ on $X$, respectively.

\begin{proposition}\label{PbStabPStab}
If $E$ is stable then $f^*E\cong \oplus_{g\in G'} g^*F$ for a destabilising stable subsheaf $F\subset f^*E$ and a suitable subset $G'\subseteq G$.
In particular, $f^*E$ is polystable.
One has
$$\Ext^k(f^*E,f^*E)\cong\left(\bigoplus_{h\in G'}\Ext^k(F,h^*F)\right)^{|G'|}$$
for all $k$, and 
$\Hom(f^*E,f^*E)\cong\IC^{|G'|\cdot |G''|}\,,$
where $G'':= \{ g\in G' \;|\; g^*F\cong F \}$. In particular, 
$f^*E$ is stable if and only if it is simple.
\end{proposition}
\begin{proof}
Let $E$ be stable. By Proposition \ref{PbStab} $f^*E$ is semistable.
Let $F\subseteq f^*E$ be a destabilising stable subsheaf.
Then by Proposition \ref{PStGshDe} $f^*E\supseteq F':=\sum_{g\in G} g^*F=\oplus_{g\in G'} g^*F$ with $G'\subseteq G$ as in Proposition \ref{PStGshDe}.
As $F'$ is a $G$-subsheaf, by \cite[\S 7 Proposition 2]{Mum70} the sheaf $(f_*F')^G$ is isomorphic to a subsheaf of $E$, and
one has $F'\cong f^*\left((f_*F')^G\right)$. Applying Lemma \ref{rHpPb} to $p(F')=p(F)=p(f^*E)$
yields $p((f_*F')^G)=p(E)\,.$
Due to the stability of $E$, one has $(f_*F')^G\cong E$, hence $F'=f^*E$.
One has
\begin{eqnarray*}
\Ext^k(f^*E,f^*E)&\cong&\bigoplus_{g,h\in G'}\Ext^k(g^*F,h^*F)\\
&\cong&\bigoplus_{g,h\in G'}\Ext^k(g^*F,h^*g^*F)\cong \left(\bigoplus_{h\in G'}\Ext^k(F,h^*F)\right)^{|G'|}
\end{eqnarray*}
for all $k$.
As $\Hom(F,h^*F)=0$ unless $F\cong h^*F$, one has
\begin{eqnarray*}
\Hom(f^*E,f^*E)\cong\left(\bigoplus_{h\in G'}\Hom(F,h^*F)\right)^{|G'|}\cong\left(\bigoplus_{h\in G''}\Hom(F,h^*F)\right)^{|G'|}\cong\IC^{|G'|\cdot |G''|}\,.
\end{eqnarray*}
In particular, if $f^*E$ is not stable, then $|G'|\neq 1$, i.e.\ $f^*E$ is not simple. Conversely, any stable sheaf is always simple.
\end{proof}

\begin{corollary}\label{PbPStabPStab}
If $E$ is polystable then $f^*E$ is a polystable $G$-sheaf.
\end{corollary}

\subsection{The pullback morphism}

We want to apply these results to moduli spaces of sheaves. We keep considering the three stability notions at once:
\begin{enumerate}
\item \textit{Gieseker stability.} The moduli space of Gieseker semistable sheaves with given Hilbert polynomial is standard by now. A detailed treatment can be found in \cite{HL10}.
\item \textit{Twisted stability.} The moduli space of twisted semistable sheaves with given twisted Hilbert polynomial is constructed in \cite[Section 4]{Yos03b}.
\item \textit{$(H,A)$-stability.} The moduli space of $(H,A)$-semistable sheaves with given Hilbert polynomial is constructed in \cite[Section 8]{Zow12}.
\end{enumerate}
Our main result in the general setting is the following:

\begin{theorem}\label{MHApb}
Let $P$ be a polynomial, $M_Y$ a quasiprojective and nonempty subscheme of the moduli space $M_Y(P)$ of semistable sheaves on $Y$ with (twisted) Hilbert polynomial $P$, $M_X=M_X(\deg f\cdot P)$ the moduli space of semistable sheaves on $X$ with (twisted) Hilbert polynomial $\deg f\cdot P$, and $M_Y^s$ and $M_X^s$ the respective loci of stable sheaves.
The pullback by $f$ induces a morphism $f^*\colon M_Y\to M_X$ which maps closed points $[E]$ to $[f^*E]$.
The closed points of its image are represented by polystable $G$-sheaves, and $(f^*)^{-1}( M_X^s )\subseteq M_Y^s\,.$
\end{theorem}
\begin{proof}
Let $F\in\Coh(Y\times S)$ be a flat family of semistable sheaves on $Y$ with (twisted) Hilbert polynomial $P$, which is parametrised by a scheme $S$.
Then $(f\times\mathrm{id}_S)^*F\in\Coh(X\times S)$ is a flat family of semistable sheaves on $X$ with (twisted) Hilbert polynomial $\deg f\cdot P$ by Proposition \ref{PbStab}.
Hence one has a natural transformation between the moduli functors, which induces a morphism $f^*\colon M_Y\to M_X$.
If $[E]$ is a closed point represented by a polystable sheaf $E$ on $Y$ then $f^*E$ is a polystable $G$-sheaf by Corollary \ref{PbPStabPStab}.
The statement on the stable locus follows also from Proposition \ref{PbStab}.
\end{proof}

For simplicity we restrict to simple cyclic coverings for the rest of this article.

\section{Cyclic coverings}\label{arbCC}

We keep all notations and assumptions as before. Additionally, for the whole section, let $f$ be a cyclic covering given by a line bundle $L$ on $Y$ of finite order $n$, and let $\nu$ be the order of $\bc_1(L)$.
Moreover, let $M_Y\subseteq M_Y(P)$ be a nonempty quasiprojective subscheme containing classes of sheaves of rank $r$ and $m\big|\frac n{\gcf(n,r)}$ such that the morphism $\varphi\colon M_Y\to M_Y$ induced by $\otimes L^{\otimes m}$ is well-defined. We denote the stable locus by $M_Y^s$.
We are interested in the following particular examples:

\begin{example}\label{Mexa}
\begin{enumerate}
\item $M_Y=M_Y(P)$ and $m=1$. As the Hilbert polynomial of a sheaf is invariant under $\otimes L$, the morphism $\varphi$ is well-defined.
\item $M_Y$ is the subscheme of $M_Y(P)$ containing classes of sheaves with fixed determinant and $m=\frac n{\gcf(n,r)}$.
\item $M_Y$ is the subscheme of $M_Y(P)$ containing classes of sheaves with fixed first Chern class and $m=\frac {\nu}{\gcf(\nu,r)}$.
\end{enumerate}
\end{example}

\noindent The following two Lemmas \ref{PreimDiv2} and \ref{PreimDiv3} ensure that $\varphi$ is well-defined in the cases 2 and 3.
\begin{lemma}\label{PreimDiv2}
Let $E$ be a coherent sheaf on $Y$. Then
$\det(E)\cong \det(E\otimes L^{j})$ if and only if $\frac n{\gcf(n,\rk E)} \big| j$.
In particular, the number of nonisomorphic sheaves of the form $E\otimes L^{j}$ with $\det(E)=\det(E\otimes L^{j})$ is equal to $\frac {\ord_E(L)\;\gcf(n,\rk E)}n$.
\end{lemma}
\begin{proof}
One has $\det(E\otimes L^{j})\cong \det(E)\otimes L^{\rk E \cdot j}$.
Hence $\det(E\otimes L^{j})\cong\det(E)$ if and only if $L^{\rk E \cdot j}\cong\cO_Y$.
This is equivalent to $n\big|\rk E\cdot j$, i.e.\ $\frac n{\gcf(n,\rk E)}\big|j$.

The nonisomorphic sheaves of the form $E\otimes L^{j}$ with $\det(E)\cong \det(E\otimes L^{j})$ are (up to isomorphisms) precisely those with $0\le j<\ord_E(L)$ and $j$ a multiple of $\frac n{\gcf(n,\rk E)}$, and there are $\frac {\ord_E(L)\,\gcf(n,\rk E)}n$ such choices.
\end{proof}

\noindent Analogously one proves the following:
\begin{lemma}\label{PreimDiv3}
Let $E$ be a coherent sheaf on $Y$. Then
$\bc_1(E)=\bc_1(E\otimes L^{j})$ if and only if $\frac {\nu}{\gcf(\nu,\rk E)} \big| j$.
In particular, the number of nonisomorphic sheaves of the form $E\otimes L^{j}$ with $\bc_1(E)=\bc_1(E\otimes L^{j})$ is equal to $\frac {\ord_E(L)\;\gcf(\nu,\rk E)}{\nu}$.
\end{lemma}
\noindent  Moreover, these two Lemmas \ref{PreimDiv2} and \ref{PreimDiv3} ensure that $\frac {\nu}{\gcf(\nu,r)}\big|\frac n{\gcf(n,r)}$, as equal determinants imply equal first Chern classes.

\subsection{The pullback morphism factorises}

The morphism between the moduli spaces induced by the pullback by $f$ factorises as follows:

\begin{proposition}\label{pbFact}
There is a commutative diagram
\begin{center}
\ \xymatrix{
M_Y \ar[rr]^{f^*} \ar@{->>}[rd]^{\pi} && M_X(nP) \\
&M_Y/\langle\varphi\rangle \ar[ru]^{\vartheta} 
}\ 
\end{center}
with $f^*(M_Y)=\vartheta(M_Y/\langle\varphi\rangle)$.
If we choose one of the Examples \ref{Mexa} then $\vartheta|_{M_Y^s/\langle \varphi \rangle}$ is injective.
\end{proposition}
\begin{proof}
As the morphism $f^*\colon M_Y\to M_X(nP)$ is $\varphi$-invariant, it factors through the quotient morphism $\pi$ and yields the morphism $\vartheta$ satisfying $f^*=\vartheta\circ\pi$.
$f^*(M_Y)=\vartheta(M_Y/\langle\varphi\rangle)$ follows from the surjectivity of $\pi$.
By Lemma \ref{PreimStab} below any preimage of $[f^*E]$ for $E\in M_Y^s$ is isomorphic to $E\otimes L^{j}$ for some $j$ with $0\le j<|G|$.
Thus $\vartheta|_{M_Y^s/\langle \varphi \rangle}$ is injective in the Examples \ref{Mexa}, using Lemmas \ref{PreimDiv2} and \ref{PreimDiv3} for the cases 2 and 3, respectively.
\end{proof}

\noindent If we choose Example \ref{Mexa}.2 and if $\gcf(n,r)=1$, then $m=n$, $\pi$ is an isomorphism, and $f^*|_{M_Y^s}$ is injective.

\begin{lemma}\label{PreimStab}
Let $E$ and $F$ be two stable sheaves on $Y$ with $f^*E\cong f^*F$.
Then there is a $j$ such that $F\cong E\otimes L^{j}$.
\end{lemma}
\begin{proof}
By Lemma \ref{CycShDec1} below one has $f_*f^*E\cong \oplus_{j=0}^{n-1}E\otimes L^{j}$ and $f_*f^*F\cong \oplus_{j=0}^{n-1}F\otimes L^{j}$.
As we assume $f^*E\cong f^*F$, one has $\oplus_{j=0}^{n-1}E\otimes L^{j}\cong f_*f^*E\cong f_*f^*F\cong\oplus_{j=0}^{n-1}F\otimes L^{j}$.
Therefore $\Hom(\oplus_{j=0}^{n-1}E\otimes L^{j},\oplus_{k=0}^{n-1}F\otimes L^k)\cong \oplus_{j,k} \Hom(E\otimes L^j, F\otimes L^k)$ is nonzero, i.e.\ for some $j,k$ there is a nontrivial homomorphism $E\otimes L^j\to F\otimes L^k$ which must be an isomorphism because $E\otimes L^j$ and $F\otimes L^k$ are stable and have the same (twisted) Hilbert polynomial. This yields the isomorphism $E\otimes L^{j-k}\to F$.
\end{proof}

\begin{lemma}\label{CycShDec1}
Let $E$ be a coherent sheaf on $Y$. Then
$f_*f^*E\cong E\otimes f_*\cO_X\cong \oplus_{j=0}^{n-1}E\otimes L^{j}\,.$
\end{lemma}
\begin{proof}
In \cite[\S 7]{Mum70} 
there is the statement that $f_*f^*E\cong E\otimes f_*\cO_X$, but without any further comment. Hence we give a proof:
Let $E_\bullet\twoheadrightarrow E$ be a locally free resolution.
As the functors $f_*$, $f^*$ and $\otimes f_*\cO_X$ 
are exact, we get the commutative diagram
\begin{center}
\ \xymatrix{
...\ar[r]& f_*f^*E_{-2}\ar[r]\ar[d]^{\cong} & f_*f^*E_{-1}\ar[r]\ar[d]^{\cong}& f_*f^*E \ar[r] &0 \\
...\ar[r]& E_{-2}\otimes f_*\cO_X\ar[r] & E_{-1}\otimes f_*\cO_X\ar[r]& E\otimes f_*\cO_X \ar[r] &0
}\ 
\end{center}
with exact rows, where the vertical arrows are the natural isomorphisms given by the projection formula.
Thus there is also an isomorphism $f_*f^*E\cong E\otimes f_*\cO_X$. 
Using \cite[Lemma I.17.2]{BHPV}, which states that $f_*\cO_X=\oplus_{j=0}^{n-1}L^{j}$, one gets the second isomorphism.
\end{proof}

\subsection{The morphism $\pi\colon M_Y\to M_Y/\langle\varphi\rangle$}

In this section we exhibit the morphism $\pi\colon M_Y\to M_Y/\langle\varphi\rangle$ from Proposition \ref{pbFact}.
We start with the following definition:

\begin{definition}\label{Eorder}
For a coherent sheaf $E$ on $Y$ the $E$-order of $L$ is the number $\ord_E(L):=\min \{ j \in \IN \;|\;  E\cong E\otimes L^{j}\}$.
\end{definition}
\noindent Clearly $1\le\ord_E(L)\le n$.

\begin{lemma}\label{PreimDiv1}
Let $E$ be a coherent sheaf on $Y$. Then
$E\otimes L^{j}\cong E\otimes L^{k}$ if and only if $\ord_E(L)|(k-j)$.
In particular, the sheaves $E\otimes L^{j}$ are pairwise nonisomorphic for $0\le j<\ord_E(L)$, and the number of nonisomorphic sheaves of the form $E\otimes L^{j}$ with $j\in\IZ$ is equal to $\ord_E(L)$.
\end{lemma}
\begin{proof}
One has $k-j=a\,\ord_E(L)+r$ for some $a,r\in\IZ$ with $0\le r<\ord_E(L)$, and 
$$E\otimes L^{k-j}\cong E\otimes L^{a\,\ord_E(L)+r}\cong E\otimes (L^{\ord_E(L)})^{a}\otimes L^{r}\cong E\otimes L^{r}\,.$$
Due to the minimality of $\ord_E(L)$ one has that $E\cong E\otimes L^{k-j}$ if and only if $r=0$.
Thus $E\otimes L^{k}\cong E\otimes L^{j}$ if and only if $r=0$, which is equivalent to $\ord_E(L)|(k-j)$.
\end{proof}

\begin{lemma}\label{PreimDiv4}
Let $E$ be a coherent sheaf on $Y$. Then
$\frac n{\gcf(n,\rk E)} \big| \ord_E(L) \big| n$. In particular, $\frac {\ord_E(L)\,\gcf(n,\rk E)} n \big| \gcf(n,\rk E)$.
\end{lemma}
\begin{proof}
This follows from Lemma \ref{PreimDiv1} with $k=n$ and $j=0$, and Lemma \ref{PreimDiv2} with $j=\ord_E(L)$.
\end{proof}

\noindent As $1\le\ord_E(L)\le n$ for a coherent sheaf $E$ on $Y$, the following definition makes sense:

\begin{definition}\label{Morder}
The $M_Y$-order of a line bundle $L$ on $Y$ of order $n$ is the number $\ord_{M_Y}(L):=\lcm \{ \ord_{E}(L) \in \IN \;|\;  [E]\in M_Y\}$, i.e. the least common multiple of all $E$-orders of $L$, where $E$ runs through all sheaves with equivalence class in $M_Y$.
\end{definition}

\begin{lemma}\label{orderDiv} For all sheaves $E$ with equivalence class in $M_Y$ one has that
$$\frac n{\gcf(n,\rk E)}\Big| \ord_E(L) \Big| \ord_{M_Y}(L) \Big|n$$ and therefore
$$\frac {\ord_E(L)\,\gcf(n,\rk E)}n  \Big| \frac {\ord_{M_Y}(L)\,\gcf(n,\rk E)}n \Big| \gcf(n,\rk E)\,.$$
\end{lemma}
\begin{proof}
By Lemma \ref{PreimDiv4}, $\frac n{\gcf(n,\rk E)} \big| \ord_E(L) \big| n$ for a coherent sheaf $E$ on $Y$.
The first row now follows from the definition of the least common multiple, and the second is an immediate consequence.
\end{proof}

\noindent We are now ready for the main result on the first factorising morphisms $\pi\colon M_Y\to M_Y/\langle\varphi\rangle$ of Proposition \ref{pbFact}:

\begin{proposition}\label{Mdecomp}
Let $M_k:=\{ [E] \in M_Y \;|\; \ord_E(L)| k \}$ for $k\in\IN$, and let $j\in\IN$ such that $M_j\neq\emptyset$.
\begin{enumerate}
\item $\varphi(M_k)=M_k$.
\item $\frac n{\gcf(n,r)}\Big|\ord_{M_j}(L)\Big|j$.
\item If $k|\ord_{M_j}(L)$, then $M_k\subseteq M_j$.
\item $M_j$ is the fixed point set of the morphism $\varphi^{\frac j m}$ (composition of morphisms).
In particular, it is closed in $M_Y$.
\item One has that $M_k=M_Y$ if and only if $\ord_{M_Y}(L)|k$.

In particular, $M_{\ord_{M_Y}(L)}=M_Y$, and $\ord_{M_Y}(L)\ge 1$ is minimal with this property.
\item $\displaystyle\Sigma_j:=\{ [E] \in M_j \;|\; \ord_E(L)\neq\ord_{M_j}(L)\}=\hspace{-0.5cm}\bigcup_{\frac n{\gcf(n,r)}|k|\ord_{M_j}(L),k\neq\ord_{M_j}(L)}\hspace{-0.5cm}M_k$, and it is closed in $M_j$ and ${M_Y}$.
\item 
Assume that $M_Y=M_Y^s$.
Then the number of preimages under the quotient morphism $\pi\colon M_j\to M_j/\langle\varphi\rangle$ of $\pi([E])$ with $[E]\in M_j$ is $\frac {\ord_E(L)}{m}$.

Moreover, if $M_j\neq \Sigma_j$, then $\pi$ is an $\frac {\ord_{M_j}(L)}{m}:1$ covering branched along $\Sigma_j$.
\end{enumerate}
\end{proposition}
\begin{proof}
\begin{enumerate}
\item Let $[E]\in M_k$. Then $\ord_{E\otimes L^{\otimes m}}(L)=\ord_{E}(L)$. Thus $\varphi([E])\in M_k$.
\item For all $[E] \in M_j$ one has that $\ord_E(L)|j$. By definition of the least common multiple, $\ord_{M_j}(L)|j$. By Lemma \ref{orderDiv} $\frac n{\gcf(n,r)}\big|\ord_{M_j}(L)$.
\item Let $[E]\in M_Y$ such that $\ord_E(L)| k$, i.e.\ $[E]\in M_k$. By assumption, $k|\ord_{M_j}(L)$, and by item 2, $\ord_{M_j}(L)|j$. Thus  $\ord_E(L)|j$, i.e.\ $[E]\in M_j$.
\item Let $[E] \in M_Y$. By Lemma \ref{PreimDiv1} $\ord_E(L)|j$ if and only if $E\cong E\otimes L^{\otimes j}$, and the morphism induced by $\otimes L^{\otimes j}$ is $\varphi^{\frac j m}\colon M_Y\to M_Y$. Here we use that $m\big|\frac n{\gcf(n,r)}$ by assumption, and $\frac n{\gcf(n,r)}\big|j$ by item 2. The fixed point set of an automorphism of a separated scheme is closed.
\item If $\ord_{M_Y}(L)|k$ then clearly $\ord_E(L)|k$ for any $[E] \in M_Y$ by definition, hence $[E]\in M_k$.
Conversely, if $M_k=M_Y$ then $\ord_M(L)|k$ by item 2.
\item
Let $[E]\in M_k$ with $k\neq\ord_{M_j}(L)$ and $\frac n{\gcf(n,r)}\big|k\big|\ord_{M_j}(L)$. By item 3 $[E]\in M_j$, and as $\ord_E(L)|k$ one has that $\ord_E(L)\le k<\ord_{M_j}(L)$.
Conversely, let $[E]\in M_j$ with $\ord_E(L)\neq\ord_{M_j}(L)$. Clearly $[E]\in M_{\ord_E(L)}$, and by Lemma \ref{orderDiv} one has that $\frac n{\gcf(n,r)}\big| \ord_E(L)\big|\ord_{M_j}(L)$.

$\Sigma_j$ is the union of finitely many subschemes $M_k$ that are closed in $M_Y$ by item 4 or because they are empty, hence $\Sigma_j$ is closed in $M_Y$, and thus also in $M_j$.
\item 
The number of nonisomorphic sheaves of the form $E\otimes L^{\otimes k}$ is $\ord_E(L)$ by Lemma \ref{PreimDiv1}.
As $m\big|\frac n{\gcf(n,r)}$ by assumption and $\frac n{\gcf(n,r)}\big|\ord_E(L)$ by Lemma \ref{orderDiv}, one has that $m|\ord_E(L)$.
Hence there are $\frac {\ord_E(L)}{m}$ nonisomorphic sheaves of the form $E\otimes L^{\otimes km}$, and their classes are the preimages of $\pi$.

If $M_j\neq \Sigma_j$, then $\Sigma_j$ is a proper closed subset. By definition of $\Sigma_j$, for $[E]\in M_j\setminus \Sigma_j$ one has that $\ord_E(L)=\ord_{M_j}(L)$.
Therefore $M_j\setminus \Sigma_j\to (M_j\setminus \Sigma_j)/\langle \varphi\rangle $ is an unbranched $\frac {\ord_{M_j}(L)}{m}:1$ covering, and $\pi$ is an $\frac {\ord_{M_j}(L)}{m}:1$ covering branched along $\Sigma_j$.
\qedhere
\end{enumerate}
\end{proof}

\noindent The statement we are most interested in is that if $\Sigma:= \{ [E] \in M_Y^s \;|\; \ord_E(L)\neq\ord_{M_Y^s}(L)\}$ is a proper subset of $M_Y^s$ then $\pi\colon M_Y^s\to M_Y^s/\langle\varphi\rangle$ is an $\frac {\ord_{M_Y^s}(L)}{m}:1$ covering branched along the closed subscheme $\Sigma$.
This follows from the above Proposition \ref{Mdecomp} using $j=\ord_{M_Y^s}(L)$.
The case of a prime power covering ensures the assumption if $M_Y^s\neq\emptyset$, as by Lemma \ref{orderDiv} $\ord_{M_Y^s}(L)$ is a prime power as well and therefore there is an $[E]\in M_Y^s$ such that $\ord_{M_Y^s}(L)=\ord_{E}(L)$. Thus we have the

\begin{corollary}\label{ppQuCov}
If $n$ is a prime power and $M_Y^s\neq\emptyset$ then $\pi\colon M_Y^s\to M_Y^s/\langle\varphi\rangle$ is an $\frac {\ord_{M_Y^s}(L)}{m}:1$ covering branched along the closed subscheme $\Sigma:= \{ [E] \in M_Y^s \;|\; \ord_E(L)\neq\ord_{M_Y^s}(L)\} $.
\end{corollary}

\subsection{Conclusions for the pullback morphism}

We now can give a description of the locus of stable sheaves becoming strictly semistable under the pullback morphism $f^*$ between the moduli spaces of sheaves. We still consider the situation of the beginning of Section \ref{arbCC} and have in mind the application to one of the cases in Example \ref{Mexa}.

\begin{theorem}\label{CycStab}
Let $M_Y^s\subseteq M_Y^s(P)$ be a nonempty quasiprojective subscheme containing classes of stable sheaves of rank $r$ such that the morphism $\varphi\colon M_Y^s\to M_Y^s$ induced by $\otimes L^{\otimes m}$ for a suitable $m\big|\frac n{\gcf(n,r)}$ is well-defined,
let $\Sigma:= \{ [E] \in M_Y^s \;|\; \ord_E(L)\neq\ord_{M_Y^s}(L)\}$, and let $f^*\colon M_Y^s\to M_X(nP)$ be the morphism induced by pullback by $f$.
The following conditions are equivalent:
\begin{enumerate}
\item $\ord_{M_Y^s}(L)=\max \{ \ord_{E}(L) \;|\;  [E]\in M_Y^s\}=n$;
\item there is a stable sheaf of the form $f^*E$ for some $E\in M_Y^s$;
\item $f^*(M_Y^s\setminus\Sigma)=f^*(M_Y^s)\cap M^s_X(nP)\neq \emptyset$;
\item $M_Y^s\setminus\Sigma=(f^*)^{-1}(M^s_X(nP))\neq \emptyset$.
\end{enumerate}
\end{theorem}
\begin{proof}
3 $\Rightarrow$ 2 is trivial.\\
2 $\Rightarrow$ 1: A stable sheaf is simple, hence by Lemma \ref{CycShDec4} $\ord_E(L)=n$.
In particular, $\ord_{M_Y^s}(L)\ge\max \{ \ord_{E}(L) \;|\;  [E]\in M_Y^s\}\ge n$.
As $\ord_E(L)|\ord_{M_Y^s}(L)|n$ by Lemma \ref{orderDiv}, one has that $n=\ord_{M_Y^s}(L)$.\\
1 $\Rightarrow$ 4: By Theorem \ref{MHApb} one has $(f^*)^{-1}(M^s_X(nP))\subseteq M_Y^s$. Thus we need to show that $E\not\in\Sigma$ if and only if $f^*E$ is stable for all $E\in M_Y^s$. 
Therefore let $E\in M_Y^s$, and assume that $\ord_{M_Y^s}(L)=n$ due to item 1.
One has that $E\not\in\Sigma$ if and only if $\ord_E(L)=\ord_{M_Y^s}(L)=n$, i.e.\ $E\not\cong E\otimes L^{\otimes j}$ for $1\le j<n$.
As $E$ and $E\otimes L$ are both stable and have the same (twisted) Hilbert polynomial, any nontrivial homomorphism is an isomorphism. Thus $E\not\cong E\otimes L^{\otimes j}$ is equivalent to $\Hom(E,E\otimes L^{\otimes j})=0$.
In particular, $\ord_E(L)=n$ if and only if $\Hom(E,E\otimes L^{\otimes j})=0$ for $1\le j<n$.
By Lemma \ref{CycShDec4} below $\Hom(E,E\otimes L^{\otimes j})=0$ for $1\le j<n$ if and only if $f^*E$ is simple.
Finally $f^*E$ is simple if and only if it is stable by Proposition \ref{PbStabPStab}.
On the other hand, the assumption $\ord_{M_Y^s}(L)=\max \{ \ord_{E}(L) \;|\;  [E]\in M_Y^s\}$ ensures that $\Sigma\neq M_Y^s$.\\
4 $\Rightarrow$ 3: This follows by taking the image by $f^*$.
\end{proof}

\noindent In the proof we used the following lemma, which will be useful more often.

\begin{lemma}\label{CycShDec4}
Let $E$ be a coherent sheaf on $Y$. Then $f^*E$ is simple if and only if $E$ is simple and $\Hom(E,E\otimes L^{\otimes j})=0$ for $1\le j<n$.
In particular, in this case, $E\not\cong E\otimes L^{\otimes j}$ for $1\le j<n$.
\end{lemma}
\begin{proof}
There is a natural isomorphism $\Hom(f^*E,f^*E)\cong\Hom(E,f_*f^*E)$, and together with Lemma \ref{CycShDec1} one gets $\Hom(f^*E,f^*E)\cong\oplus_{j=0}^{n-1}\Hom(E,E\otimes L^{\otimes j})$.
The claim now follows because $\Hom(E,E)\neq 0$.
\end{proof}

\noindent Coprimeness of rank and $n$ yield the simplest cases for the pullback morphism:

\begin{theorem}\label{gcdnr1}
Let $M_Y^s\subseteq M_Y^s(P)$ be a nonempty quasiprojective subscheme containing classes of stable sheaves of rank $r$ and $m|n$ such that the morphism $\varphi\colon M_Y^s\to M_Y^s$ induced by $\otimes L^{\otimes m}$ is well-defined, and let $f^*\colon M_Y^s\to M_X(nP)$ be the morphism induced by pullback by $f$.
Assume that $\gcf(n,r)=1$. Then $\pi\colon M_Y^s\to M_Y^s/\langle\varphi\rangle$ is an unbranched $\frac {n}{m}:1$ covering and $f^*(M_Y^s)\subseteq M_X^s(nP)$.

If we are in one of the cases of Example \ref{Mexa} then the morphism $f^*\colon M_Y^s\to M_X(nP)$ is $\frac {n}{m}:1$ onto its image.
\end{theorem}
\begin{proof}
By Lemma \ref{PreimDiv4}, one has that $\ord_E(L)=n$ for all $[E]\in M_Y^s$ and therefore $\ord_{M_Y^s}(L)=n$.
Thus $\Sigma:= \{ [E] \in M_Y^s \;|\; \ord_E(L)\neq\ord_{M_Y^s}(L)\}=\emptyset$, and by Proposition \ref{Mdecomp} $\pi$ is an unbranched $\frac {n}{m}:1$ covering.
As item 1 of Theorem \ref{CycStab} holds, item 3 of that theorem yields $f^*(M_Y^s)\subseteq M_X^s(nP)$.

If we are in one of the cases of Example \ref{Mexa} then by Proposition \ref{pbFact} the morphism ${M_Y^s/\langle \varphi \rangle}\to M_X(nP)$ induced by $f^*$ is injective. Hence the composition $M_Y^s\to{M_Y^s/\langle \varphi \rangle}\to M_X(nP)$ is $\frac {n}{m}:1$ onto its image.
\end{proof}

\section{Canonical coverings of surfaces}\label{Surfaces}

In this section we apply our results to canonical coverings of surfaces, i.e.\ cyclic coverings given by a torsion canonical bundle.
We start with some definitions in order to fix notations and conventions.

\subsection{Notations and conventions}

Let $Y$ be a nonsingular projective irreducible surface over $\IC$, $H$ an ample divisor on $Y$, and $E$ a coherent sheaf on $Y$. 
We associate the element
$$u(E):=(\rk E,\bc_1(E),\chi(E))\in \Hev:=\IN_0\oplus\NS(Y)\oplus\IZ$$
of sheaf invariants to $E$.
We avoid the elegant notion of a Mukai vector in favour of keeping torsion inside $\NS(Y)$.
For an element $u:=(r,c,\chi)\in \Hev$ we define
\begin{eqnarray*}
P(u)&:=&r \frac{H^2}2n^2+\left(c-r\;\frac{K_Y}2\right).Hn+\chi\\
\Delta(u)&:=&c^2 - 2r\chi+2r^2\chi(\cO_Y)-rc.K_Y\\
\chi(u,u)&:=&\chi(\cO_Y) r^2-\Delta(u)\\
f^*u&:=&(r,f^*c,\deg f\;\chi)
\end{eqnarray*}
If $E$ satisfies $u(E)=u$, then, by Riemann-Roch, its Hilbert polynomial is $P(u)$, its discriminant\footnote{Be aware of different conventions of the discriminant's definition.} is $\Delta(u)$, $u(f^*E)=f^*u$ and
$$\chi(E,E):=\sum_{k=0}^2 \ext^k(E,E) = \chi(u,u)\,,$$ where $\ext^k(E,E):=\dim \Ext^k(E,E)$. We will also write $\hom(E,F):=\dim\Hom(E,F)$ for two coherent sheaves $E,F$.

\subsection{Symplectic moduli spaces and Lagrangian subvarieties}

We now restrict to canonical coverings, and thus assume that $Y$ has a torsion canonical bundle $K_Y$ of order, say $n$. Let $f\colon X\to Y$ be the covering given by $K_Y$. Then $K_X$ is trivial and either $X$ is a K3 surface, $Y$ is an Enriques surface and $n=2$, or $X$ is an abelian surface, $Y$ is a bielliptic surface and $n=2,3,4$ or $6$, see e.g.\ \cite{BHPV}.

As in Section \ref{arbCC} let $\nu$ be the order of $\bc_1(K_Y)$. If $Y$ is an Enriques surfaces, one has that $\nu=n=2$ due to \cite[Proposition 15.2]{BHPV}.
On the other hand, the situation for an elliptic surface $Y$ is more complicated. By \cite[\S 1]{Ser90} $\bH^2(Y,\IZ)$ may or may not be torsion free depending on the type of $Y$. In particular, $\nu$ may be smaller than $n$, and if $\bH^2(Y,\IZ)$ is torsion free then $\bc_1(K_Y)=0$ and $\nu=1$.

Let $u:=(r,c,\chi)\in \Hev$, and let $H$ and $A$ be two ample divisors on $Y$.
We will consider only Gieseker stability and $(H,A)$-stability in the following, and for the latter we assume $r>0$ for simplicity, as is done in \cite[Section 6]{Zow12}.
Let $M_Y(u)$ be the moduli space of Gieseker $H$- or $(H,A)$-semistable sheaves $E$ with $u(E)=u$. It is projective and a subscheme of $M_Y(P(u))$, hence the results of Section \ref{arbdim} apply. Moreover, by setting $m:=\frac {\nu}{\gcf(\nu,r)}$ we are in the situation of Example \ref{Mexa}.3 with $M_Y=M_Y(u)$.
In case we need to distinguish between Gieseker stability and $(H,A)$-stability, we will use $M_{Y;H}(u)$ and $M_{Y;H,A}(u)$, respectively.
Recall that $M_{Y;H,H}(u)=M_{Y;H}(u)$.
We denote the open subscheme of $M_Y(u)$ of stable sheaves on $Y$ by $M_Y^s(u)$ and the smooth locus of $M_Y^s(u)$ by $M_Y^{sm}(u)$.

$M^s_X(f^*u)$ is nonsingular, each connected component has dimension $2-\chi(f^*u,f^*u)$ and it carries a symplectic form due to Mukai \cite{Muk84}.

The morphism $f^*\colon M_Y(u)\to M_X(\deg f\cdot P)$ induced by the pullback by $f$ which is described in Theorem \ref{MHApb} has image inside $M_X(f^*u)$.
Thus we can replace $M_X(\deg f\cdot P)$ by $M_X(f^*u)$ in the results of Section \ref{arbdim} and consider $f^*$ as a morphism $M_Y(u)\to M_X(f^*u)$.

\begin{proposition}\label{pbsymvan}
The pullback of the symplectic form on $M^s_X(f^*u)$ by the restricted morphism $M_Y^{sm}(u)\cap (f^*)^{-1}(M^s_X(f^*u))\stackrel {f^*}\to M^s_X(f^*u)$ vanishes.
\end{proposition}
\begin{proof}
This is proven in the proof of \cite[item 3 of the main theorem in \S 3]{Kim98a}, for $X$ a K3 surface and $\mu$-stable sheaves. His proof works as well for Gieseker stability and $(H,A)$-stability, and for $X$ an abelian surface.
\end{proof}

\begin{corollary}\label{ressymvan}
The restriction of the symplectic form on $M^s_X(f^*u)$ to $f^*(M_Y^{sm}(u))\cap M^s_X(f^*u)$ vanishes.
\end{corollary}

\noindent This leads to the question whether the subvariety $f^*(M_Y^{sm}(u))\cap M^s_X(f^*u)$ is a Lagrangian subvariety of $M^s_X(f^*u)$. Hence we need to calculate dimensions.

\begin{proposition}\label{CCfEst}
Let $E$ be a coherent sheaf on $Y$ with $u(E)=u$ such that $f^*E$ is stable. Then 
$E$ is stable,
$E\not\cong E\otimes K_Y^j$ for $1\le j<n$ and
$\ext^2(E,E)=0$.
$M_Y(u)$ is nonsingular in $[E]$ of expected dimension
\begin{eqnarray*}
\dim_{E} M_Y(u)&=&\ext^1(E,E)\,,\quad\textrm{and}\\
\dim_{f^*E} M_X(f^*u)&=&\ext^1(f^*E,f^*E)=2-n+n \dim_{E} M_Y(u)\,.
\end{eqnarray*}
In particular, one has that
$\ext^1(E,E)=0$  if and only if $\ext^1(f^*E,f^*E)=0$, and that
${\dim_{f^*E} M_X(f^*u)=2\dim_{E} M_Y(u)}$ if and only if $n=2$ or $\ext^1(E,E)=1$.
\end{proposition}
\begin{proof}
$E$ is stable by Proposition \ref{PbStab}. In particular, $E$ is simple and by Lemma \ref{CycShDec4} one has $E\not\cong E\otimes K_Y^j$ for $1\le j<n$.
$f^*E$ is simple as well, thus by the following Lemma \ref{CycCanDim} 
\begin{eqnarray}
\ext^1(f^*E,f^*E)=2-n+n\; \ext^1(E,E) \label{ext1fEfE}
\end{eqnarray}
and $\ext^2(E,E)=0$. Therefore $M_Y(u)$ is nonsingular in $[E]$ of expected dimension $\ext^1(E,E)$.
As $X$ is K3 or abelian, $M^s_X(f^*u)$ is nonsingular of expected dimension $2-\chi(f^*u,f^*u)=\ext^1(f^*E,f^*E)$ as stated already above.
Equation (\ref{ext1fEfE}) yields the two equivalences, as $\ext^1$ is always nonnegative and $n\ge 2$.
\end{proof}

\noindent In the proof we used the following 

\begin{lemma}\label{CycCanDim}
Let $E$ be a coherent sheaf on $Y$. Then
$$\ext^1(f^*E,f^*E)=2\hom(f^*E,f^*E)-n\left( \hom(E,E)+\hom(E,E\otimes K_Y)\right)+n\; \ext^1(E,E)\,.$$
If $f^*E$ is simple, then 
$\ext^1(f^*E,f^*E)=2-n+n\; \ext^1(E,E)\,,$
and
$\ext^2(E,E)=0$.
\end{lemma}
\begin{proof}
We first prove that $\chi(f^*E,f^*F)=\deg f\; \chi(E,F)$ for two coherent sheaves $E$ and $F$, which holds also if $f$ is not cyclic and $Y$ has higher dimension.
If $E$ is locally free, then $f^*E$ is locally free as well, and one has that
\begin{eqnarray*}
\chi(f^*E,f^*F)&=&\chi( \cH om(f^*E,f^*F) )\\
&=&\chi( f^*\cH om(E,F) )\\
&=&\deg f\; \chi( \cH om(E,F) )=\deg f\; \chi(E,F)\,, 
\end{eqnarray*}
where we used the canonical isomorphism
$f^*\cH om(E,F) \cong \cH om(f^*E,f^*F)$
\cite{EGA1} (6.7.6) and \cite[\S 12 Theorem 2]{Mum70}.
If $E$ is not locally free, consider a locally free resolution $E_\bullet\twoheadrightarrow E$, which gives a locally free resolution $f^*E_\bullet\twoheadrightarrow f^*E$, and use the additivity of $\chi$.
Thus the above claim holds. As $Y$ is a surface, by Serre duality one has
\begin{eqnarray*}
\ext^1(f^*E,f^*E)&=&\hom(f^*E,f^*E)+\hom(f^*E,f^*E\otimes K_X)\\
&&-n\left( \hom(E,E)+\hom(E,E\otimes K_Y)\right)+n\; \ext^1(E,E)\\
&=&2\hom(f^*E,f^*E)-n\left( \hom(E,E)+\hom(E,E\otimes K_Y)\right)+n\; \ext^1(E,E)\;,
\end{eqnarray*}
whilst the last equation holds due to the fact that the canonical bundle $K_X$ is trivial.\\
If $f^*E$ is simple, then Lemma \ref{CycShDec4} yields that $E$ is simple and $\hom(E,E\otimes K_Y)=0$. 
Hence
$$\ext^1(f^*E,f^*E)=2-n+n\; \ext^1(E,E)\,,$$
and by Serre duality one has $\ext^2(E,E)=0$.
\end{proof}

\noindent The two main results of this section now are contained in the following Theorem \ref{CC} and in Theorem \ref{CCsst}.

\begin{theorem}\label{CC}
Let $Y$ be an Enriques or bielliptic surface, $f\colon X\to Y$ the covering given by the torsion canonical bundle $K_Y$ of order $n$ on $Y$,
$\nu$ the order of $\bc_1(K_Y)$,
$u=(r,c,\chi)\in \Hev$, $M_Y(u)$ the moduli space of Gieseker or $(H,A)$-semistable sheaves $E$ with $u(E)=u$ and $M_Y^s(u)$ the open subscheme of $M_Y(u)$ of stable sheaves on $Y$.
Assume that there is a coherent sheaf $E$ on $Y$ with $u(E)=u$ such that $f^*E$ is stable and let
$\Sigma:=\{ [E]\in M^s_Y(u) \;|\; \ord_E(K_Y)\neq n\}$\footnote{See Definition \ref{Eorder}.}.
Then 
\begin{enumerate}
\item $M^s_Y(u)\setminus\Sigma$ is a nonempty and nonsingular open subset of $M^s_Y(u)$,
\item $f^*(M^s_Y(u)\setminus\Sigma)=f^*(M_Y(u))\cap M^s_X(f^*u)$,
\item $f^*(\Sigma)\cap M^s_X(f^*u)=\emptyset$,
\item $f^*$ induces a $\frac{n\;\gcf(\nu,r)}\nu:1$ covering $M^s_Y(u)\to f^*(M^s_Y(u))$ branched along $\Sigma$, and
\item $\dim M^s_X(f^*u)=2-n+n \dim f^*(M^s_Y(u)\setminus\Sigma)$.
\end{enumerate}
In particular, $f^*(M^s_Y(u)\setminus\Sigma)$ is a Lagrangian subvariety of $M^s_X(f^*u)$ if and only if $n=2$ or $\dim f^*(M^s_Y(u)\setminus\Sigma)=1$.
\end{theorem}
\begin{proof}
A coherent sheaf $E$ on $Y$ with $u(E)=u$ such that $f^*E$ is stable is stable itself by Proposition \ref{PbStab}. In particular, $M^s_Y(u)\neq\emptyset$.
Let $m=\frac \nu{\gcf(\nu,r)}$ as at the beginning of this section, where $\nu$ is still the order of $\bc_1(K_Y)$. Recall that we are in the situation of Example \ref{Mexa}.3 with $M_Y=M_Y(u)$.
Item 2 of Theorem \ref{CycStab} holds by assumption, hence its other items hold as well, where we replace $M_X(nP)$ by $M_X(f^*u)$ as mentioned above:
\begin{itemize}
\item The $M_Y^s$-order of $K_Y$ (Definition \ref{Morder}) is $n$,
\item $f^*(M_Y^s\setminus\Sigma)=f^*(M_Y)\cap M^s_X(f^*u)\neq \emptyset$, and
\item $M_Y^s\setminus\Sigma=(f^*)^{-1}(M^s_X(f^*u))\neq \emptyset$. In particular, $f^*(\Sigma)\cap M^s_X(f^*u)=\emptyset$.
\end{itemize}
Therefore items 1-3 of our claim hold.

As all sheaves with class in $f^*(M^s\setminus\Sigma)$ are stable, Proposition \ref{CCfEst} yields that $M^s\setminus\Sigma$ is nonsingular and $\dim M^s_X(f^*u)=2-n+n \dim (M^s\setminus\Sigma)$.
By Proposition \ref{pbFact} and Proposition \ref{Mdecomp} item 7 with $j=n$ the morphism $f^*$ induces a $\frac{n\;\gcf(\nu,r)}\nu:1$ covering $M^s\to f^*(M^s)$ branched along $\Sigma$.
In particular, $\dim (M^s\setminus\Sigma)=\dim f^*(M^s\setminus\Sigma)$.

By Corollary \ref{ressymvan} the restriction of the symplectic form on $M^s_X(f^*u)$ to $f^*(M^s\setminus\Sigma)$ vanishes, thus
$f^*(M^s\setminus\Sigma)$ is a Lagrangian subvariety if and only if $\dim M^s_X(f^*u)=2\dim f^*(M^s\setminus\Sigma)$.
This is equivalent to $n=2$ or $\dim f^*(M^s_Y(u)\setminus\Sigma)=1$ due to item 5.
\end{proof}

If $f^*(M^s_Y(u))\subseteq M^s_X(f^*u)$ --- e.g.\ if $\gcf(n,r)=1$ (use Theorem \ref{gcdnr1}) or if $f^*u$ is primitive and $f^*H$ (or $f^*A$, respectively) is general (see the appendix) --- and if $M^s_Y(u)$ is nonempty, then the theorem holds with $\Sigma=\emptyset$.
In particular, the covering $M^s_Y(u)\stackrel {f^*} \to f^*(M^s_Y(u))$ is unbranched.

If $f^*(M^s_Y(u))\cap M^s_X(f^*u)=\emptyset$, i.e.\ if any stable sheaf becomes strictly semistable after pullback, then this construction does not seem to yield any Lagrangian subvariety, as $f^*(M^s_Y(u))$ is outside the locus where the symplectic form is defined.

Theorem \ref{CC} is concerned with the case that there is a coherent sheaf $E$ on $Y$ with $u(E)=u$ such that $f^*E$ is stable.
This means that in particular, $E$ is already stable by Proposition \ref{PbStab}.
We now want to consider the opposite case, i.e.\ $E$ is stable but for all such $E$ the pullback $f^*E$ is not stable.

\begin{theorem}\label{CCsst}
Let $Y$ be an Enriques or bielliptic surface, $f\colon X\to Y$ the covering given by the torsion canonical bundle $K_Y$ of order $n$ on $Y$,
$\nu$ the order of $\bc_1(K_Y)$,
$u=(r,c,\chi)\in \Hev$, $M_Y(u)$ the moduli space of Gieseker or $(H,A)$-semistable sheaves $E$ with $u(E)=u$ and $M_Y^s(u)$ the open subscheme of $M_Y(u)$ of stable sheaves on $Y$.
Assume that $M^s_Y(u)\neq\emptyset$ and $f^*(M^s_Y(u))\cap M^s_X(f^*u)=\emptyset$.
Moreover, let $\ell$ be the $M^s_Y(u)$-order of $L$ defined in Definition \ref{Morder}, and $m:=\frac \nu{\gcf(\nu,r)}$ as in Example \ref{Mexa}.3.
Then
\begin{enumerate}
\item $m<n$;
\item for all $E\in M^s_Y(u)$ one has that $E\cong E\otimes K_Y^{\ell}$;
\item if $n$ is a prime power then $\ell\neq n$ and $f^*$ induces an $\frac {\ell}{m}:1$ covering $M^s_Y(u)\to f^*(M^s_Y(u))$ branched along the closed subscheme $\Sigma:= \{ [E] \in M^s_Y(u) \;|\; \ord_E(L)\neq\ell\}$.
\end{enumerate}
\end{theorem}
\begin{proof}
\begin{enumerate}
\item By Lemma \ref{PreimDiv4} $m|\ord_E(L)|n$ for all $[E]\in M^s_Y(u)$. Assume that $\max \{ \ord_{E}(L) \;|\;  [E]\in M^s_Y(u)\}=n$. Then $\ell=n$, and item 3 of Theorem \ref{CycStab} contradicts our assumption. Hence $\ord_{E}(L)<n$ for all $[E]\in M^s_Y(u)$, and, in particular, $m<n$.
\item This follows from the definition of the $M^s_Y(u)$-order of $L$.
\item Let $n$ be a prime power. By Lemma \ref{orderDiv} $\ell|n$, thus $\ell$ is a prime power as well, and therefore there is an $[E]\in M_Y^s(u)$ such that $\ell=\ord_{E}(L)$. Assume that $\ell=n$. Then, as in item 1, item 3 of Theorem \ref{CycStab} contradicts our assumption. Hence $\ell\neq n$. The second statement is Proposition \ref{pbFact} together with Corollary \ref{ppQuCov}.\qedhere
\end{enumerate}
\end{proof}

\noindent If $n$ is not a prime power, which means in our case that $n=6$, it is not ensured that $\Sigma$ is a proper subset.

\subsection{Lagrangian subvarieties in irreducible symplectic manifolds}\label{LIHS}

We keep all assumptions.
For $n=\nu=2$, in particular, for a K3 surface covering an Enriques surface, Theorem \ref{CC} immediately simplifies to

\begin{proposition}\label{CanDouSt}
Let $Y$ be an Enriques or bielliptic surface with torsion canonical bundle $K_Y$ of order $2$ on $Y$, $f\colon X\to Y$ the covering given by $K_Y$,
$u=(r,c,\chi)\in \Hev$, $M_Y(u)$ the moduli space of Gieseker or $(H,A)$-semistable sheaves $E$ with $u(E)=u$ and $M_Y^s(u)$ the open subscheme of $M_Y(u)$ of 
stable sheaves on $Y$.
Assume that the order of $\bc_1(K_Y)$ is 2 as well and that there is a coherent sheaf $E$ on $Y$ with $u(E)=u$ such that $f^*E$ is stable.
\begin{enumerate}
\item If the rank $r$  is odd, then one has the following:
\begin{enumerate}
\item $M^s_Y(u)$ is a nonempty and nonsingular,
\item $f^*(M^s_Y(u))=f^*(M_Y(u))\cap M^s_X(f^*u)$,
\item $f^*$ induces an isomorphism $M^s_Y(u)\to f^*(M^s_Y(u))$,
\item $\dim M^s_X(f^*u)=2 \dim f^*(M^s_Y(u))$, and
\item $f^*(M^s_Y(u))$ is a Lagrangian subvariety of $M^s_X(f^*u)$.
\end{enumerate}
\item If the rank $r$ is even and $\Sigma:=\{ E\in M^s_Y(u) \;|\; E\otimes K_Y\cong E \}$, then
\begin{enumerate}
\item $M^s_Y(u)\setminus\Sigma$ is a nonempty and nonsingular open subset of $M^s_Y(u)$,
\item $f^*(M^s_Y(u)\setminus\Sigma)=f^*(M_Y(u))\cap M^s_X(f^*u)$ and $f^*(\Sigma)\cap M^s_X(f^*u)=\emptyset$,
\item $f^*$ induces a $2:1$ covering $M^s_Y(u)\to f^*(M^s_Y(u))$ branched along $\Sigma$,
\item $\dim M^s_X(f^*u)=2 \dim f^*(M^s_Y(u)\setminus\Sigma)$, and
\item $f^*(M^s_Y(u)\setminus\Sigma)$ is a Lagrangian subvariety of $M^s_X(f^*u)$.
\end{enumerate}
\end{enumerate}
\end{proposition}

\noindent  If $f^*u$ is primitive and $f^*H$ (or $f^*A$, respectively) is $f^*u$-general, then $M_X(f^*u)=M^s_X(f^*u)$ and $\Sigma=\emptyset$ also for even rank.
Whenever $M_Y(u)$ is nonempty, it yields the Lagrangian subvariety $f^*(M_Y(u))$ in the symplectic manifold $M^s_X(f^*u)$. Note that the latter is an irreducible (holomorphically) symplectic manifold if $X$ is a K3 surface and, when considering $(H,A)$-stability, if a numerical assumption on the invariants holds, see \cite[Corollary 6.7]{Zow12}.

On the other hand, in the case $n=\nu=2$, Theorem \ref{CCsst} simplifies to

\begin{proposition}\label{CanDouSSSt}
Let $Y$ be an Enriques or bielliptic surface with torsion canonical bundle $K_Y$ of order $2$ on $Y$, $f\colon X\to Y$ the covering given by $K_Y$,
$u=(r,c,\chi)\in \Hev$, $M_Y(u)$ the moduli space of Gieseker or $(H,A)$-semistable sheaves $E$ with $u(E)=u$ and $M_Y^s(u)$ the open subscheme of $M_Y(u)$ of 
stable sheaves on $Y$.
Assume that the order of $\bc_1(K_Y)$ is 2 as well, that $M^s_Y(u)\neq\emptyset$ and that $f^*(M^s_Y(u))\cap M^s_X(f^*u)=\emptyset$.
Then the rank $r$ is even, for all $E\in M^s_Y(u)$ one has that $E\cong E\otimes K_Y$, and $f^*|_{M^s_Y(u)}$ is injective.
\end{proposition}

Similar results have been proven by Kim in \cite{Kim98a} for $\mu$-stable sheaves of positive rank on Enriques surfaces.\\

As a first question one might ask when the assumption $M^s_Y(u)\neq\emptyset$ is satisfied. Recall that $M^s_Y(u)\neq\emptyset$ is a necessary condition also for Proposition \ref{CanDouSt} due to Proposition \ref{PbStab}.
A general case of odd rank sheaves has been considered in \cite{Yos03}. \cite[Theorem  4.6]{Yos03} implies that
if $Y$ is an unnodal Enriques surface, i.e.\ there is no $-2$-curve, $u$ is primitive with $\chi(u,u)\ge -1$ and $H$ is $u$-general, then $M^s_{Y;H}(u)$ is nonempty and irreducible.
The assumption also implies that $M_{Y;H}(u)=M^s_{Y;H}(u)$, i.e.\ there are no strictly semistable sheaves.\\

If $X$ is a K3 surface, then $\Pic^0(X)=0$, and $M_X(f^*u)$ is an irreducible symplectic manifold if $f^*u$ is primitive and $f^*H$ is $f^*u$-general (or $f^*A$ is $f^*u$-general, respectively, and additionally the above-mentioned numerical condition holds).
If $X$ is an abelian surface, the situation is different: in order to produce higher dimensional irreducible symplectic manifolds one has to get rid of superfluous factors in the Bogomolov decomposition by taking a fibre of the Albanese map. Hence we will now fix the determinant of the considered sheaves and additionally reduce the moduli space to the kernel of a suitable summation map.

Let $Y$ be a bielliptic surface and $f\colon X\to Y$ the canonical covering. In particular, $X$ is an abelian surface and $f$ is cyclic of order $n=2,3,4$ or $6$.
Let still $H$ and $A$ be two ample divisors on $Y$, and we continue considering Gieseker stability and $(H,A)$-stability in the following.

We associate the element $$w(E):=(\rk E,\det(E),\chi(E))\in \Hpev:=\IN_0\oplus\Pic(Y)\oplus\IZ$$ of sheaf invariants to the sheaf $E$.
We fix an element $w:=(r,d,\chi)\in \Hpev$ and define $$P(w):=r \frac{H^2}2n^2+\bc_1(d).Hn+\chi\,.$$
If $E$ satisfies $w(E)=w$, then its Hilbert polynomial is $P(w)$.

Let $M_Y(w)$ be the moduli space of Gieseker or $(H,A)$-semistable sheaves $E$ with $w(E)=w$. It is projective and a subscheme of $M_Y(P(w))$, hence the results of Section \ref{arbdim} apply again. Moreover, by setting $m:=\frac n{\gcf(n,r)}$ we are in the situation of Example \ref{Mexa}.2 with $M_Y=M_Y(w)$.
In case we need to distinguish between Gieseker stability and $(H,A)$-stability, we will use $M_{Y;H}(w)$ and $M_{Y;H,A}(w)$, respectively.
Recall again that $M_{Y;H,H}(w)=M_{Y;H}(w)$.
We denote the open subscheme of $M_Y(w)$ of stable sheaves on $Y$ by $M_Y^s(w)$. 
$M^s_X(f^*w)$ is nonsingular, each connected component has dimension $-\chi(f^*w,f^*w)$, where, of course,  $\chi((r,d,a),(r,d,a)):=\chi((r,\bc_1(d),a),(r,\bc_1(d),a))$ for $(r,d,a)\in\Hpev$.
The morphism $f^*\colon M_Y(w)\to M_X(\deg f\cdot P(w))$ induced by the pullback by $f$ which is described in Theorem \ref{MHApb} has image inside $M_X(f^*w)$.
Thus we can analogously replace $M_X(\deg f\cdot P(w))$ by $M_X(f^*w)$ in the results of Section \ref{arbdim} and consider $f^*$ as a morphism $M_Y(w)\to M_X(f^*w)$.

\begin{proposition}\label{CCfEst0}
Let $E$ be a coherent sheaf on $Y$ with $w(E)=w$ such that $f^*E$ is stable.
Then $E$ is stable, $E\not\cong E\otimes K_Y^j$ for $1\le j<n$ and $\ext^2(E,E)_0=0$.
$M_Y(w)$ is nonsingular in $[E]$ of expected dimension
\begin{eqnarray*}
\dim_{E} M_Y(w)&=&\ext^1(E,E)_0=\ext^1(E,E)-1\,,\quad\textrm{and}\\
\dim_{f^*E} M_X(f^*w)&=&\ext^1(f^*E,f^*E)_0=\ext^1(f^*E,f^*E)-2=n \dim_{E} M_Y(w)\,.
\end{eqnarray*}
\end{proposition}
\begin{proof}
By Proposition \ref{CCfEst}, $E$ is stable, $E\not\cong E\otimes K_Y^j$ for $1\le j<n$, $\ext^2(E,E)=0$ and 
\begin{eqnarray}
\ext^1(f^*E,f^*E)&=&2-n+n \ext^1(E,E)\,. \label{CCfEst0_1}
\end{eqnarray}
Hence, in particular, $\ext^2(E,E)_0=0$.
By \cite[Theorem 4.5.4]{HL10} and its immediate generalisation to $(H,A)$-stability, $M_Y(w)$ is nonsingular in $[E]$ and of expected dimension $\ext^1(E,E)_0$,
and $M_X(f^*w)$ is nonsingular in $[f^*E]$ and of expected dimension $\ext^1(f^*E,f^*E)_0$.
By the formula after \cite[Corollary 4.5.5]{HL10} one has that
\begin{eqnarray*}
\ext^1(E,E)_0&=&\chi(O_Y)-\chi(E,E)=0-1+\ext^1(E,E)-0 \quad\mathrm{and}\quad\\
\ext^1(f^*E,f^*E)_0&=&\chi(O_X)-\chi(f^*E,f^*E)=0-1+\ext^1(f^*E,f^*E)-1\\
&\stackrel{(\ref{CCfEst0_1})}=&n (\ext^1(E,E)-1)= n\ext^1(E,E)_0\,. 
\end{eqnarray*} \qede
\end{proof}

\noindent We can now adjust Theorems \ref{CC} and \ref{CCsst} to the modified situation:

\begin{theorem}\label{CC0}
Assume that there is a coherent sheaf $E$ on $Y$ with $w(E)=w$ such that $f^*E$ is stable and let $\Sigma:=\{ [E]\in M^s_Y(w) \;|\; \ord_E(K_Y)\neq n\}$.
Then 
\begin{enumerate}
\item $M^s_Y(w)\setminus\Sigma$ is a nonempty and nonsingular open subset of $M^s_Y(w)$,
\item $f^*(M^s_Y(w)\setminus\Sigma)=f^*(M_Y(w))\cap M^s_X(f^*w)$,
\item $f^*(\Sigma)\cap M^s_X(f^*w)=\emptyset$,
\item $f^*$ induces a $\gcf(n,r):1$ covering $M^s_Y(w)\to f^*(M^s_Y(w))$ branched along $\Sigma$, and
\item $\dim M^s_X(f^*w)=n \dim f^*(M^s_Y(u)\setminus\Sigma)$.
\end{enumerate}
\end{theorem}
\begin{proof}
The proof goes analogous to Theorem \ref{CC} using Proposition \ref{CCfEst0} instead of Proposition \ref{CCfEst}.
\end{proof}

\noindent Analogously to Theorem \ref{CCsst} one proves 

\begin{theorem} 
Assume that $M^s_Y(w)\neq\emptyset$ and $f^*(M^s_Y(w))\cap M^s_X(f^*w)=\emptyset$. Moreover, let $\ell$ be the $M^s_Y(w)$-order of $L$, and $m:=\frac n{\gcf(n,r)}$ as in Example \ref{Mexa}.2. Then
\begin{enumerate}
\item $m<n$ and $\gcf(n,r)\neq 1$,
\item For all $E\in M^s_Y(w)$ one has that $E\cong E\otimes K_Y^{\ell}$, and
\item if $n$ is a prime power then $\ell\neq n$ and $f^*$ induces an $\frac {\ell}{m}:1$ covering $M^s_Y(w)\to f^*(M^s_Y(w))$ branched along the closed subscheme $\Sigma:= \{ [E] \in M^s_Y(w) \;|\; \ord_E(L)\neq\ell\}$.
\end{enumerate}
\end{theorem}
\noindent Recall that if $n$ is not a prime power, which means in our case that $n=6$, it is not ensured that $\Sigma$ is a proper subset.

We want to reduce further, which needs some preparation.
A short introduction to bielliptic surfaces can be found e.g.\ in \cite[V.5]{BHPV}. We need the following:
Every bielliptic surface $Y$ admits a finite \'etale covering $B\times C\to Y$ factorising via $X$, where $B$ and $C$ are elliptic curves.
$Y\cong (B\times C)/G$, where $G\subset C$ is a finite subgroup acting on $B$ such that $B/G\cong \IP^1$.
One has that $G\cong \IZ/(\IZ/n)\times \IZ/(\IZ/m)$ with $n$ still the order of $K_Y$ and the possibilities $m=1$ for any $n$, $m=2$ if $n=2$ or $4$, and $m=3$ only for $n=3$.
The generator of the group $\IZ/(\IZ/n)$ acts on $B$ by multiplication with $e^{2\pi i/n}$, and the generator of the group $\IZ/(\IZ/m)$ (in the case $m\neq 1$) by translation by some $a\in B$ with certain properties, see e.g.\ \cite[V.5]{BHPV}.
The covering abelian surface is given by $X\cong (B\times C)\Big/\IZ/(\IZ/m)$, and the group structure on $B\times C$ descends to the group structure of the abelian surface $X$.
This group structure induces a summation map $\sum \colon \CH_0(X)\to X$, where $\CH_0(X)$ is the Chow group of $X$. 
The second Chern class associates an element $\bc_2^{CH}(E)\in \CH_0(X)$ to any coherent sheaf $E$ on $X$.

Recall that one considers the kernel $K_X(f^*w)$ of the morphism $$A_{f^*w} \colon M_X(f^*w)\to X, [E]\mapsto \sum \bc_2^{CH}(E)$$ in order to get rid of less interesting factors in the Beauville-Bogomolov decomposition of $M_X(f^*w)$.
In particular, $K_{X;f^*H}(f^*w)$ is an irreducible symplectic manifold if $f^*w$ is primitive, $f^*H$ is $f^*w$-general and $-\chi(f^*w,f^*w)\ge 6$ \cite[Theorem 0.2]{Yos01}.
As $$f^*(M_Y(w))\cap K_X(f^*w) = f^*( (f^*)^{-1}(K_X(f^*w)))\,,$$ we are interested in $K_Y(w):=(f^*)^{-1}(K_X(f^*w))$ and the morphism $K_Y(w)\to K_X(f^*w)$.

For a sheaf class $[E]\in M_Y(w)$ one has that $\bc_2^{CH}(f^*E) = f^*\bc_2^{CH}(E)$.
Let us write $\bc_2^{CH}(E)=\sum_ia_i[b_i,c_i]_Y$ with $a_i\in \IZ$, $b_i\in B$ and $c_i\in C$, where $[\bullet]_Y$ denotes the image of $\bullet\in B\times C$ under the quotient morphism $B\times C\to Y$. This enables us to calculate
\begin{eqnarray*}
\sum \bc_2^{CH}(f^*E) &=& \sum f^*\bc_2^{CH}(E)\\
&=& \sum \sum_ia_i f^*[b_i,c_i]_Y\\
&=&\sum_ia_i \sum_{k=0}^{n-1} [\rho^k b_i,c_i+kg]_X\\
&=&\sum_ia_i [0,nc_i+\frac{n(n-1)}2g]_X\,,
\end{eqnarray*}
where $[\bullet]_X$ denotes the image of $\bullet\in B\times C$ under the quotient morphism $B\times C\to X$.
Therefore the image of $A_{f^*w}\circ f^*$ is at most one-dimensional, and
$$\codim_{M_Y(w)} K_Y(w) = \codim_{f^*(M_Y(w))} f^*(K_Y(w)) \le 1\,.$$
As the pullback of the symplectic structure to the smooth locus of $M_Y(w)$ vanishes by Proposition \ref{pbsymvan}, the corresponding restrictions to $K_Y(w)$ and to $f^*(K_Y(w))$ vanish as well.
We are interested in Lagrangian subvarieties of higher dimensional irreducible symplectic manifolds, so we assume now that $n=2$, $-\chi(f^*w,f^*w)\ge 6$, $f^*w$ is primitive and $f^*H$ is $f^*w$-general.
The codimension of $K_{X;f^*H}(f^*w)$ in $M_{X;f^*H}(f^*w)$ is 2, hence intersecting $f^*(M_{Y;H}(w))$ with $K_{X;f^*H}(f^*w)$ has to reduce the dimension at least by 1, i.e.\
$$\codim_{M_{Y;H}(w)} K_{Y;H}(w) = \codim_{f^*(M_{Y;H}(w))} f^*(K_{Y;H}(w)) \ge 1\,.$$
Thus we have proven the following:

\begin{proposition}\label{CanDouStK}
Let $n=2$, i.e.\ the canonical covering of the bielliptic surface $Y$ by the abelian surface $X$ has degree 2.
If $f^*w$ is primitive, $f^*H$ is $f^*w$-general and $-\chi(f^*w,f^*w)\ge 6$, then the image of the morphism $K_{Y;H}(w)\stackrel{f^*}\to K_{X;f^*H}(f^*w)$ is a Lagrangian subvariety.
\end{proposition}

\noindent This result should generalise to $(H,A)$-stability:

\begin{conjecture}
Let $n=2$. If $f^*w$ is primitive, $f^*A$ is $f^*w$-general and $-\chi(f^*w,f^*w)\ge 6$, then $K_{X;f^*H,f^*A}(f^*w)$ is an irreducible symplectic manifold and the image of the morphism $K_{Y;H,A}(w)\stackrel{f^*}\to K_{X;f^*H,f^*A}(f^*w)$ is a Lagrangian subvariety.
\end{conjecture}

\subsection{Examples}\label{example}

As described so far, canonical double coverings produce Lagrangian subvarieties via pullback.
There are basically two different cases to distinguish: odd and even rank.
We expect that the moduli space of sheaves of odd rank on an Enriques or bielliptic surface $Y$ behave like the rank 1 case.
In particular, Yoshioka proved that $M_H(r,0,\frac{r+1}2)\cong \Hilb^{\frac{r+1}2}(Y)$ \cite[Corollary 4.4]{Yos03} for odd rank $r$ and general $H$.
As any moduli space of sheaves of rank 1 and of fixed determinant is isomorphic to a Hilbert scheme of points on the underlying surface, we have a look at these Hilbert schemes.

The easiest case is the morphism $f^*\colon Y\to \Hilb^2(X)$ induced by pullback by a canonical double covering $f$, which embeds an Enriques or bielliptic surface into the Hilbert scheme of 2 points of a K3 or abelian surface, respectively. More generally, one has the embedding $f^*\colon \Hilb^t(Y)\to \Hilb^{2t}(X)$ by Proposition \ref{CanDouSt} item 1 (Enriques case) and Theorem \ref{CC0} (bielliptic case). First one might ask what is the geometry of $\Hilb^t(Y)$? Oguiso and Schr\"oer proved in \cite{OS11} the following results:

\begin{theorem}[\cite{OS11} 3.1] Let $Y$ be an Enriques surface and $m\ge 2$. Then $\pi_1(\Hilb^m(Y))$ is cyclic of order two, and the universal covering of $\Hilb^m(Y)$ is a Calabi-Yau manifold.
\end{theorem}

\begin{theorem}[\cite{OS11} 3.5] Let $Y$ be a bielliptic surface and $m\ge 2$. Then there is an \'etale covering $\cH \to\Hilb^m(Y)$ so that $\cH$ is the product of an elliptic curve and a Calabi-Yau manifold of dimension $2m-1$.
\end{theorem}

\noindent Thus the subvariety $f^*(\Hilb^t(Y))\subset\Hilb^{2t}(X)$, which is isomorphic to $\Hilb^t(Y)$, has the corresponding covering from the respective theorem above.
We ask the following

\begin{question}\label{OSgen}
How do these results of Oguiso and Schr\"oer generalise to moduli spaces of sheaves?
\end{question}

On the other hand, we expect that the moduli spaces of sheaves of even rank are less close to the Hilbert scheme case.
Quite recently Hauzer established a connection between certain moduli spaces of sheaves of even rank and certain moduli spaces of sheaves of rank 2 or 4 if the Enriques surface is unnodal \cite[Theorem 2.8]{Hau10}. Moreover, he described particular one-dimensional moduli spaces of rank 2 sheaves on Enriques surfaces:

\begin{theorem}[\cite{Hau10} 0.1]
Let $Y$ be an Enriques surface, and $F_1$ and $F_2$ the two multiple fibres of an elliptic fibration of $Y$.
Then there exists an explicit class of polarisations $H$ such that $M_{Y;H}(2,F_1,1)\cong F_2$.
\end{theorem}
\noindent Be careful that in Hauzer's notation, the last entry of the triple $(2,F_1,1)$ in his article is the second Chern class. However, in this case, Riemann-Roch yields $\chi =1= \bc_2$, hence they look the same in our notation.

This choice of invariants yields the morphism $f^*\colon M_{Y;H}(2,F_1,1)\to M_{X;f^*H}(2,f^*F_1,2)$, which maps $F_2$ 2:1 onto a Lagrangian subvariety $L$ in the surface $M_{X;f^*H}(2,f^*F_1,2)$.

As Hauzer explains in the proof of \cite[Lemma 1.1]{Hau10}, $F_1$ is indivisible in $\bH^2(Y,\IZ)_0$, the torsion free part of $\bH^2(Y,\IZ)$.
The induced map $f^*\colon \bH^2(Y,\IZ)_0\to\bH^2(X,\IZ)$ is injective, hence $f^*F_1$ is primitive in $f^*(\bH^2(Y,\IZ))\subset \bH^2(X,\IZ)$.
By \cite[Proposition 2.3]{Nam85} $f^*(\bH^2(Y,\IZ))$ is the $+1$ eigenspace of the covering involution and therefore a primitive sublattice of $\bH^2(X,\IZ)$.
Hence $f^*F_1$ is primitive in $\bH^2(X,\IZ)$ as well.
If we now choose $H$ such that $f^*H$ is general, then $H$ must already be general by Proposition \ref{fHgenHgen}.
This ensures that we only have stable sheaves, i.e.\ $M_{Y;H}(2,F_1,1)=M^s_{Y;H}(2,F_1,1)$ and $S:=M_{X;f^*H}(2,f^*F_1,2)=M^s_{X;f^*H}(2,f^*F_1,2)$. In particular, $S$ is a projective K3 surface.

Let us consider a general Enriques surface in the sense of \cite[Proposition 5.6]{Nam85}, i.e.\ one has that $f^*(\NS(Y))=\NS(X)$. In particular, hyperplanes in $\NS(X)$ have hyperplanes as preimages in $\NS(Y)$. Going through the proof of \cite[Lemma 1.1]{Hau10} one checks that under this assumption Hauzer's choice of $H$ allows to choose $H$ such that $f^*H$ is general as well.
As $M_{X;f^*H}(2,f^*F_1,2)=M_{X;f^*H}^s(2,f^*F_1,2)$, by Proposition \ref{CanDouSt} item 2 the morphism $f^*$ induces an unramified covering $F_2\to L$ of degree 2 and $M_{Y;H}(2,F_1,1)\cong F_2$ is nonsingular elliptic.
Hence the Lagrangian subvariety $L$ is a nonsingular elliptic curve as well by the Hurwitz formula.

\section{Outlook}\label{Outlook}

Although we have quite general results on the pullback morphism between moduli spaces, the application to particular situations is more interesting if one has relevant results at least on one of these moduli spaces.

One classical example of a cyclic covering is the Godeaux surface covered by the Fermat quintic.
However, not very much is known on the moduli space of semistable sheaves if the underlying surface is of general type. Results of Li \cite{Li94} and, slightly generalised by O'Grady \cite{OGr97}, show that in general the moduli space is of general type as well. General means in particular that the second Chern class of the sheaves is very large.

There are some recent results by Mestrano and Simpson \cite{MS11} on the moduli space $M_{Y;\cO_X(1)}(2,\bc_1(\mathcal O_X(-1)),\chi)$ with arbitrary $\chi$ and $Y\subset\IP^3$ a very general quintic surface. The choice of the first Chern class ensures that the four notions of Gieseker/slope (semi)stability coincide.
By very general the authors mean smooth and at least that $\mathrm{Pic}(X) \cong \mathrm{Pic}(\IP^3) = \IZ$, with further genericity conditions where necessary.

Unfortunately the Fermat quintic is not very general in this sense. The article \cite{Schue11} contains plenty of concrete examples of quintic surfaces in $\IP^3$. The Fermat quintic is contained as Example 3 and is shown to have Picard number 37.

\begin{appendix}
\section{General ample divisors}
In this appendix we recall the notion of general ample divisors and state two results concerning generality and pullback.\\

Let the situation be as in Section \ref{Surfaces}. The ample cone of $Y$ carries a chamber structure for a given triple $u=(r,c,\chi)\in\Hev$ of invariants.
The definition depends on $r$. In the case of $r=1$ we agree that the whole ample cone is the only chamber.

For $r>1$, we follow the definition in \cite[Section 4.C]{HL10}.
Let $\Num(Y):=\Pic(Y)/\equiv$, where $\equiv$ denotes numerical equivalence, and $\Delta:=\Delta(u)>0$.
\begin{definition}
Let
$$W(r,\Delta):=\{ \xi^\perp \cap \Amp(Y)_{\IQ} \;|\; \xi\in\Num(Y) \quad\mathrm{with}\quad -\frac {r^2}4 \Delta \le \xi^2 < 0 \}\,,$$
whose elements are called $u$-walls.
The connected components of the complement of the union of all $u$-walls are called $u$-chambers.
An ample divisor is called $u$-general if it is not contained in a $u$-wall.
\end{definition}

\noindent  The set $W(r,\Delta)$ is locally finite in $\Amp(Y)_{\IQ}$ by \cite[Lemma 4.C.2]{HL10}.

For $r=0$, we follow the definition in \cite[Section 1.4]{Yos01}.

\begin{definition}
Let $c\neq 0$ be effective. For every sheaf $E$ with $u(E)=u$ and every subsheaf $F\subseteq E$ we define $L:=\chi(F)\bc_1(E)-\chi(E)\bc_1(F)$, and for $L\ne 0$ we call
$$W_L:=L^\perp \cap\Amp(Y)_\IQ$$ the $u$-wall defined by $L$.
The connected components of the complement of the union of all $u$-walls are called $u$-chambers.
An ample divisor is called $u$-general if it is not contained in a $u$-wall.
\end{definition}

If $r=0=\chi$ then the notion of $H$-(semi)stability for a sheaf $E$ with $u(E)=u$ is independent of the choice of $H$ and one cannot introduce the notion of a $u$-general ample divisor in this particular case.
However, we can move away from this case, as tensoring with the ample line bundle $H$ yields the isomorphism
$M_{Y;H}(0,c,\chi)\cong M_{Y;H}(0,c,\chi+c.H)\,.$
Thus one can assume without loss of generality that $\chi\neq 0$ when investigating the moduli spaces of one-dimensional semistable sheaves on a surface.

\begin{lemma}\label{DeltaPb}
$\Delta(f^*u)=\deg f \Delta(u)$.
\end{lemma}
\begin{proof}\ \\\vspace{-1.3cm}
\begin{eqnarray*}
\Delta(f^*u)
&=& (f^*c)^2 - 2r\deg f\;\chi+2r^2\chi(\cO_X)-rf^*c.K_X\\
&=& \deg f\;c^2 - 2r\deg f\;\chi+2r^2\deg f\;\chi(\cO_Y)-rf^*c.f^*K_Y\\
&=& \deg f(c^2 - 2r\chi+2r^2\chi(\cO_Y)-rc.K_Y)= \deg f\;\Delta(u)
\end{eqnarray*}\qede
\end{proof}

\begin{proposition}\label{fHgenHgen}
If $f^*H$ is $f^*u$-general, then $H$ is $u$-general.
\end{proposition}
\begin{proof}
According to the definition of a general ample divisor, one has to distinguish by the rank $r$.
\begin{enumerate}
\item \emph{Positive rank $r\ge 2$, i.e.\ twodimensional sheaves:}
Let $\xi\in\Num(Y)$ with $-\frac {r^2}4 \Delta \le \xi^2 < 0$.
Then $f^*\xi\in\Num(X)$, and one has $(f^*\xi)^2=\deg f\;\xi^2$.
Hence $$-\frac {r^2}4 \deg f\;\Delta \le \deg f\; \xi^2=(f^*\xi)^2 < 0\,.$$ 
By Lemma \ref{DeltaPb} one has $\Delta(f^*u)=\deg f \Delta(u)$, so $f^*\xi$ defines the $f^*u$-wall $(f^*\xi)^\perp \cap \Amp(X)_{\IQ}$.
As $f^*H$ is $f^*u$-general, one has $0\neq f^*\xi.f^*H=\deg f\;\xi.H$.
Thus $H$ is $u$-general.
\item \emph{Rank $r=0$ with effective first Chern class $c$, i.e.\ onedimensional sheaves:}
Let $E$ be a coherent sheaf with $u(E)=u$ and $F\subseteq E$ such that $L:=\chi(F)\bc_1(E)-\chi(E)\bc_1(F)\ne 0$.
One has 
\begin{eqnarray*}
\tilde L
&:=&\chi(f^*F)\bc_1(f^*E)-\chi(f^*E)\bc_1(f^*F)\\
&=&\deg f\;\chi(F)f^*\bc_1(E)-\deg f\;\chi(E)f^*\bc_1(F)=\deg f\;f^*L\,.
\end{eqnarray*}
As $f^*H$ is $f^*u$-general, one has $0\neq \tilde L.f^*H=\deg f\; f^*L.f^*H=(\deg f)^2\;L.H$.
Thus $H$ is $u$-general.
\qedhere
\end{enumerate}
\end{proof}

\begin{lemma}
If $f^*u$ is primitive, then $u$ is primitive.
\end{lemma}
\begin{proof}
Let $u=mu_0$ with $u_0=(r_0,c_0,\chi_0)\in\Hev$ and $m\in\IN$.
Then $$f^*u=(mr_0,f^*(mc_0),\deg f\;m\chi_0)=mf^*u_0\,.$$
As $f^*u$ is primitive, one has $m=1$, i.e.\ $u$ is primitive.
\end{proof}

\end{appendix}
 
\bibliographystyle{amsalpha}
\bibliography{my}

\end{document}